\documentclass[letterpaper, 11pt]{amsart} 
\usepackage{mabliautoref}
\usepackage{comment}

\usepackage[all]{xy}

\usepackage{amsmath,amssymb,amscd,amsthm,amsfonts}

\usepackage{enumerate}

\usepackage{layout}

\usepackage[usenames,dvipsnames]{xcolor}

\usepackage{hyperref}

\hypersetup{
bookmarks,
bookmarksdepth=3,
bookmarksopen,
bookmarksnumbered,
pdfstartview=FitH,
colorlinks,backref,hyperindex,
linkcolor=Sepia,
anchorcolor=BurntOrange,
citecolor=MidnightBlue,
citecolor=OliveGreen,
filecolor=BlueViolet,
menucolor=Yellow,
urlcolor=OliveGreen
}

\DeclareMathOperator{\Supp}{Supp}

\numberwithin{equation}{section}
\newcommand{\M}{\mathcal{M}}
\newcommand{\A}{\mathcal{A}}
\renewcommand{\O}{\mathcal{O}}

\newcommand{\g}{\overline{\gamma}}
\newcommand{\q}{\overline{q}}

\newcommand{\E}{\mathcal{E}}

\newcommand{\Q}{\mathbb{Q}}
\newcommand{\R}{\mathbb{R}}

\renewcommand{\A}{\mathcal{A}}

\newcommand{\D}{\Delta}

\renewcommand{\P}{\mathbb{P}}

\newcommand{\Z}{\mathbb{Z}}
\newcommand{\N}{\mathbb{N}}
\renewcommand{\O}{\mathcal{O}}
\newcommand{\cF}{\mathcal{F}}

\newcommand{\cL}{\mathcal{L}}

\DeclareMathOperator{\Bl}{Bl}

\DeclareMathOperator{\Spec}{Spec}

\DeclareMathOperator{\Aut}{Aut}

\DeclareMathOperator{\Sym}{Sym}

\DeclareMathOperator{\pr}{pr}
\DeclareMathOperator{\Ex}{Ex}

\DeclareMathOperator{\sing}{sing}

\DeclareMathOperator{\Exc}{Exc}

\DeclareMathOperator{\rk}{rk}

\DeclareMathOperator{\codim}{codim}

\DeclareMathOperator{\lct}{lct}
\DeclareMathOperator{\vol}{vol}
\DeclareMathOperator{\coeff}{coeff}
\DeclareMathOperator{\im}{Im}
\DeclareMathOperator{\red}{red}

\newcommand{\sHom}{\mathcal{H}om}

\newcommand{\Bir}{\mathrm{Bir}}

\theoremstyle{plain}

\theoremstyle{definition}

\usepackage{tikz}									
\usetikzlibrary{matrix}
\usetikzlibrary{patterns}
\usetikzlibrary{decorations.pathmorphing}
\usetikzlibrary{cd}
\usetikzlibrary{calc}
\usetikzlibrary{decorations.markings,shapes,positioning}

\usepackage{xparse}
\usepackage{minibox}
\usepackage{mathtools}

\NewDocumentCommand\UpArrow{O{2.0ex} O{black}}{%
   \mathrel{\tikz[baseline] \draw [line width=0.5pt, decoration={markings,mark=at position 1 with {\arrow[scale=2, line width=0.25pt]{to}}},  postaction={decorate}, #2] (0,0) -- ++(0,#1);}
}

\newcommand{\expl}[2]{\underset{\mathclap{\minibox[c]{$\UpArrow[10pt]$\\ \fbox{\footnotesize #2}}}}{#1}}

\newcommand{\explshift}[3]{\underset{\mathclap{\minibox[c]{$\UpArrow[10pt]$\\ \hspace{#1} \fbox{\footnotesize #3}}}}{#2}}

\newcommand{\explpar}[3]{\underset{\mathclap{\minibox[c]{$\UpArrow[10pt]$\\  \fbox{\parbox{#1}{\footnotesize #3}}}}}{#2}}

\newcommand{\explparshift}[4]{\underset{\mathclap{\minibox[c]{$\UpArrow[10pt]$\\ \hspace{#2} \fbox{\parbox{#1}{\footnotesize #4}}}}}{#3}}

\renewcommand{\theenumi}{\arabic{enumi}}

\title{Effective positivity of Hodge bundles and applications}
\date{\today}

\usepackage{enumitem}
\setlist[itemize,enumerate]{leftmargin=0.8cm,itemsep=3pt,topsep=3pt}

\author{Giulio Codogni}
\address{Dipartimento di Matematica, Universit\`a degli Studi di Roma Tor Vergata, Via della Ricerca Scientifica, 00133 Roma, Italy.}
\email{codogni@mat.uniroma2.it}

\author{Zsolt Patakfalvi}
\address{
EPFL\\
SB MATHGEOM CAG  \\
MA B3 635 (B\^atiment MA) \\
Station 8 \\
CH-1015 Lausanne}
\email{zsolt.patakfalvi@epfl.ch}

\author{Luca Tasin}
\address{Dipartimento di Matematica F.\ Enriques, Universit\`a degli Studi di Milano, Via Cesare Saldini 50, 20133 Milano, Italy} 
\email{luca.tasin@unimi.it}


\begin{document}

\begin{abstract}
We prove new boundedness results across different areas of algebraic geometry, stemming from a unifying technical starting point: bounding the integer $q > 0$ such that the $q$-th Hodge bundle becomes (semi-)positive for families of stable varieties. 

This result allows us to show that  for stable families $f: X \to T$ of maximal variation with klt general fiber and relative dimension $n$ there exist the following bounds:  
\begin{enumerate}
\setlength{\itemsep}{0pt}
\item a lower bound for the Chow-Mumford volume $\left( \lambda_{CM,f} \right)^{\dim T}$ of the form $\delta^{\dim T}$, where $\delta$ is uniform; 
\item
  a uniform lower bound on $K_{X/T}^{n+1}$, when $T$ is a curve; 
  \item an upper bound for  $|\Aut(f)|$ when $T$ is a curve, depending uniformly linearly on $K_{X/T}^{n+1}$.
\end{enumerate}
Additionally, we draw several consequences on the subspaces of the moduli space of stable varieties parameterizing at least one klt variety, such as the positivity of Hodge bundles and a lower bound on the Chow-Mumford volume in terms of the dimension and the volume of the parametrized varieties (the volume is needed only if working on the coarse moduli space).

We also give pair versions of the above results with coefficients varying in a DCC set.
\end{abstract}

\maketitle

\section{Introduction}
\begin{center}
    
\emph{We work over an algebraically closed field of characteristic 0.}
\end{center}

\subsection{Motivation}  
In this paper, we address several boundedness questions that arise in the study of families of stable varieties and their moduli spaces. While these problems originate from different areas — birational geometry, moduli theory, Hodge theory, foliation theory — they all turn out to be governed by a common geometric theme: the positivity of Hodge bundles in families of stable varieties. Specifically, we consider the following:

\begin{enumerate}
    \item On the moduli stack $\mathcal{M}_{n,v}$  of stable varieties of dimension $n$ and volume $v$, the assignment
    \begin{equation*}
        \mathcal{M}_{n,v}(T) \ni 
    \{f: X\to T \}  \mapsto f_* \O_X(qK_{X/T})
    \end{equation*}
    defines a coherent sheaf.  When is this a  nef/ample vector bundle?
    
    Note the following:
    \begin{itemize}
        \item This coherent sheaf is usually called the $q$-th Hodge bundle, and despite its name, it is not at all evident a priori even that it is a vector bundle, see \autoref{prop:base_change_stable_families}.
        \item The nefness and ampleness for $q$ divisible enough is known, even in the log-case \cite{Fuj,KP}. This is how the projectivity of the coarse moduli space of stable pairs has been shown, and a reason why this sort of statements are cornerstone in the theory of moduli spaces. We are asking here for a version where $q$ is made effective, hopefully by allowing almost all positive $q$.
        \item Similar (semi-)positivity statements and their applications have been widespread, and thus central  in algebraic geometry. Giving a thorough overview in this direction is beyond the scope of the present article, some of them are surveyed in \cite[Section 6.3.E]{Laz2} and in \cite{Horing}, a non-exhausting list of references is \cite{EV, EV90, Fuj16, Fuj, FF14, FFS14, Fujita, Kaw81, Kaw82, Kol86a,Kol86b, Kol87, Kol90, KP, LPS20, LS24, PS, Vie82, Vie83}.
    \end{itemize} 
    \item On the same moduli space the CM (Chow-Mumford) line bundle is 
    $$
    \lambda_{CM} = (n+1)! \lim_{q \to \infty} \frac{\lambda_q}{q^{n+1}},
    $$ 
    where $\lambda_q$ is the determinant of the vector bundle from the previous point. This is known to be ample on the coarse moduli space \cite{PX}, so one can ask the lower bound for its volume on any irreducible component. Let us recall that, over the complex numbers, the Chow-Mumford line bundle is related to the existence of K\"{a}hler-Einstein metrics, and its volume is related to the Weil-Petersson metric on the moduli space defined out of K\"{a}hler-Einstein metrics (see  e.g. \cite{Weil-Petersson} and references therein) .

    \item Is there a lower bound for the volume of fibrations of fixed relative dimension with klt (Kawamata log terminal) general fibres? (see \autoref{q:volume} for a precise formulation in the set-up of foliations). This is the relative variant of a lower bound that is well-known for varieties \cite{HMX} and which is an  essential step in showing the celebrated boundedness results for varieties of general type/stable varieties \cite{HMX-bound}.
    
    \item Given a family $f: X \to T$ of general type varieties of maximal variation, is there a linear bound in fixed dimension for the automorphism group of the fibration by the volume $\vol(K_{X/T})$ of the fibration? This is again the relative version of a  famous result for varieties \cite{HMX-Aut}.
  
\end{enumerate}
Below, we explain our positive answers to the above 4 boundedness questions. 

\subsection{Effective positivity results for Hodge bundles}

Positivity statements usually stem from stability assumptions - often the same stability assumptions used to construct moduli spaces. Let us briefly recall the stability notions used in this paper. A projective pair is stable if it has semi-log canonical (slc for short) singularities and the log canonical divisor is ample; over the complex numbers, this notion is well-known to be equivalent to the existence of a K\"{a}hler-Einstein metric with prescribed singularities along the boundary \cite[Theorem B]{BG14}. A flat fibration $f\colon (X,\Delta)\to B$ over a variety $B$ is locally stable if the fibers of $f$ are demi-normal, $\Delta$ is an effective $\Q$-divisor, $\Supp \Delta$ does not contain any fiber, and singular codimension $1$ point (i.e., the node) of any fiber, $K_{X/B}+ \Delta$ is $\Q$-Cartier, and $\big(K_{X_b}, \Delta_b \big)$ is slc for every $b \in B$. It is stable if in addition $K_{X/B}+ \Delta$ is $f$-ample. This definition is as in \cite[Definition-Theorem 4.7]{Kobook}, except that we consider only $\Q$-Cartier divisors rather than $\R$-Cartier. We refer to \autoref{S:notations} for further notations.

\medskip

The following result on ampleness and bigness is an effective version of \cite[Theorem 7.1]{KP} (and \cite{Kol87, Kol90}). 
The part on nefness is a generalisation of \cite{Fuj}. Note that in general for a $\Q$-divisor $D$, one defines $\O_X(D) := \O_X(\lfloor D \rfloor)$. 

\begin{theorem}
\label{thm:main_posivity}
Let $f\colon (X,\Delta)\to B$ be a stable family over a normal projective base and let $q>0$ be an integer. 

Then 
\begin{enumerate}
    \item {\scshape [Nefness] }
    \label{itm:main_posivity:nef}(\autoref{cor:semi_posit_lambda_e_CM})
    $f_* \O_X(q(K_{X/B}+ \Delta))$  is nef if one of the following conditions holds
    \begin{itemize}
        \item  $B$ is a curve, or
        \item $q\Delta$ is a $\Z$-divisor, or 
        \item for every normal projective curve $T$ and every morphism $g\colon T \to B$,  the base change 
        $$g^*f_* \O_X(q(K_{X/B}+ \Delta))\cong h_*\O_{X_T}(q(K_{X_T/T}+ \Delta_T))$$ 
        holds, where $h\colon (X_T,\Delta_T)\to T$ is the base change of $f$.
    \end{itemize}
    
    \item {\scshape [Bigness]} \label{itm:main_posivity:det_big} (\autoref{thm:postivity_lambda}) $\det f_* \O_X(q(K_{X/B}+ \Delta))$ and (\autoref{thm:bigness}) $f_*\O_X(2q(K_{X/B}+\D))$ are big, if all of the following conditions are satisfied:
    \begin{itemize}
        \item there is  $b \in B$ such that $(X_b,\Delta_b)$ is klt,
        \item $b \mapsto (X_b, \Delta_b)$ has maximal variation, and 
        \item $q\geq 2$, $q\Delta$  is a $\Z$-divisor and $0<\rk f_* \O_X(q(K_{X/B}+ \Delta))$.
    \end{itemize} 
    \item {\scshape [Ampleness]} \label{itm:main_posivity:det_ample} (\autoref{thm:postivity_lambda}) $\det f_* \O_X(q(K_{X/B}+ \Delta))$ is ample, if all of the following conditions are satisfied: 
    \begin{itemize}
        \item   $(X_b, \Delta_b)$ is klt for all $b \in B$, and 
        \item the isomorphism equivalence classes are finite, that is, for every closed point $b \in B$, there are at most finitely many closed points $s \in B$ such that $(X_b, \Delta_b) \cong (X_s, \Delta_s)$, and
         \item $q\geq 2$, $q\Delta$ is a $\Z$-divisor and $\rk f_* \O_X(q(K_{X/B}+ \Delta))>0$.
    \end{itemize}
\end{enumerate} 
\end{theorem}

Let us comment on the hypotheses of Item \autoref{itm:main_posivity:nef}. The curve case is a generalization of \cite[Theorems 1.10 and 1.11]{Fuj}. It is proved in  \autoref{thm: Fujino} with a delicate reduction to \cite{Fuj}. Note that \autoref{thm: Fujino} is slightly more general than what is claimed in Item \autoref{itm:main_posivity:nef} as it allows special fibers with worse than slc singularities. The assumption $q\Delta$ integral guarantees the base change property requested in the third item, see \autoref{prop:base_change_stable_families}. That base change is not always true, some counterexamples are give in \cite{Kol18}, see also our \autoref{ex:basechange}. In \cite[Theorem 1]{Kol18} it is shown that the base change holds when the generic fiber is normal and the coefficients are greater than one half, the one half case requires some care as discussed in the Warning 1.1. from loc. cit. For the non-normal case see \cite[Question 8]{Kol18}. We think it would be interesting to further investigate when the formation of Hodge bundles commutes with base change.

\autoref{ex:q0integral} shows that the assumption $q\Delta$ integral is necessary for points \autoref{itm:main_posivity:det_big} and \autoref{itm:main_posivity:det_ample} of \autoref{thm:main_posivity} to hold. The  reason is that the variation might come from the boundary, which might not contribute to the sheaf $f_* \O_X(q(K_{X/T}+ \Delta))$. In \autoref{ex:q0=1} it is shown that the assumption $q\geq 2$ is also necessary. Concerning the question whether the klt assumptions of \autoref{thm:main_posivity} are necessary, see \autoref{rem:sharpness}.

\begin{example}\label{ex:q0integral}
Set $X= E \times E$, where $E$ is an elliptic curve, $f = \pr_2 : X \to T=E$ the second projection, and let $r>1$ be an integer. Let $R \subseteq E$ be a finite subset containing $r-1$ distinct closed points, and let $\Gamma \subseteq X$ be the diagonal. Set $\Delta = \frac{1}{r} \big( \Gamma +  R \times T \big) $. Then, $f_* \O_X\big(r' (K_{X/T} + \Delta) \big) \cong \O_T $ for every $r' <r$. 
\end{example}

In fact, the positivity statements of \autoref{thm:main_posivity} for determinants can be also bootstrapped to positivity statements for the actual sheaves (earlier effective results can be found in \cite{EV, EV90} for the Gorenstein case, while non-effective statements are contained in \cite{KP}):

\begin{corollary}[= \autoref{cor:ample}]\label{corintro:ample} 
Let  $f\colon (X,\Delta)\to T$ be a  fibration where $(X,\D)$ is a projective normal log-pair with general fibre klt and $T$ is a smooth projective curve. Assume that $K_{X/B}+\D$ is $f$-ample and that $f$ has maximal variation. Let $q\ge 2$ be an integer such that $q\D$ is integral. Then
 $$
 f_*\O_X(q(K_{X/B}+\D))
 $$
 is an ample vector bundle on $T$ whenever it is non-zero.
\end{corollary}

\begin{remark}
\label{rem:sharpness}
Remarks on the sharpness of  \autoref{corintro:ample}:
\begin{itemize} 
    \item The klt assumption for the general fibre is necessary, as shown in \cite[Example 7.5]{KP}. The same question could arise also for points \autoref{itm:main_posivity:det_big} and \autoref{itm:main_posivity:det_ample} of \autoref{thm:main_posivity}. However, in that case we do not know the answer. We do use the klt assumption multiple times during the proof, e.g. to apply \autoref{p_nef_Weil} or to ensure that the invariant $\g_q$ in \autoref{thm:lowermu_basic} is strictly positive.

     \item The same conclusion as in \autoref{corintro:ample} holds assuming $(X,\D)$ klt and $K_{X/B}+\D$ $f$-big instead of $f$-ample, since going to a relative canonical model does not change the Hodge bundles (\cite[Corollary 1.25]{Kosing}).
\end{itemize}
\end{remark}

\subsubsection{Methods of proof} 

Let us first explain the proof of nefness from \autoref{S:positivity}. The base case is when $q=1$, the boundary is integral and the total space is simple normal crossing; here nefness is a consequence of Hodge theory, and we take it for granted (the most general versions are in \cite{Fuj}). In \autoref{S:positivity}, we generalize the argument from \cite{Fuj} to bootstrap the statement for any $q$ and singular total spaces from the base case.

 Our first new result about positivity is \autoref{thm:lowermu_basic}: given a family over a curve with some technical assumptions, we give an effective lower bound on the lowest slope of the Harder-Narasimhan filtration of the Hodge bundles. This bound is in terms of the Berman-Gibbs invariant discussed in \autoref{ss:BGV divisor}. The proof builds on the Viehweg trick. This result generalizes earlier works by Esnault and Viehweg \cite{EV90, EV, EV92}.  It has consequences also about the nefness threshold of the relative log canonical bundle, as explained in \autoref{nef_threshold}.

To conclude, we combine \autoref{thm:lowermu_basic} with earlier non-effective positivity results such as the ones from \cite{Kol87, Kol90, KP, PX} to obtain our new effective bounds. These last results rely on the so called ampleness lemma, which alone can not give effective statements.

Along the way, we use some new technical base change results, and a careful analysis on how Hodge bundles are affected by the stable reduction, see \autoref{prop:base_change_stable_families} and \autoref{prop:stable_reduction}.

\subsection{Lower bound on the Chow-Mumford volume of fibrations}\label{s:intro:fib}

Combining the positivity results explained above together with the theory of slope inequalities from \cite{CTV}, see also \autoref{S:slope}, we are able to obtain the following lower bounds on the volume of the Chow-Mumford divisor $\lambda_{CM}$ of a stable family. We say that a set is DDC if it satisfies the descending chain condition.

\begin{theorem}[= \autoref{thm:volumelambda}]\label{thmintro:volumelambda}
Let $n$ be a positive integer and $\Lambda \subset \mathbb Q \cap [0,1]$ be a DCC set. Then there exists a constant $\delta=\delta(n,\Lambda) >0$ such that 
$$
\lambda_{CM}^{d}\geq \delta^{d}
$$
for any stable family $f\colon (X,\Delta)\to B$	satisfying the following conditions: $B$ is a normal projective variety of dimension $d\geq 1$, the coefficients of $\Delta$ are in $\Lambda$, the relative dimension of $f$ is $n$,  $f$ has maximal variation, and at least one fiber is klt.

In particular, if $d=1$, we obtain that
$$
(K_{X/B}+\D)^{n+1} \ge \delta.
$$
\end{theorem}

A novelty in our approach is that, given a family $X \to B$, finite base changes (that would change the volume) are only used to prove positivity of the Hodge bundles on $B$, and then  we use a slope inequality from \cite{CTV} recalled in \autoref{slopeinequality} to get the bound on the volume. 

\begin{question}
Given a stable family   $f\colon (X,\Delta)\to B$ of relative dimension $n$, with $d=\dim B \ge 2$,  and coefficient of $\Delta$ in a DCC set $\Lambda$, is there a relation between $(K_{X/B} + \D)^{n+d}$ and $\lambda_{CM}^d$ only in terms of $d$, $n$ and $\Lambda$?
\end{question}

\subsection{Volume of fibrations and foliations}

For varieties of general type, a landmark result of Alexeev in dimension two and of Hacon–McKernan–Xu in arbitrary dimension establishes that their volumes satisfy the DCC, and in particular, are bounded below by a positive constant depending only on the dimension. More precisely:
	
\begin{theorem}{\cite[Theorem 1.3]{HMX}}\label{thm:HMX}
 Fix a positive integer $n$ and a DCC set $I \subset [0,1]$.  Let $\mathcal D$ be the set of log canonical pairs $(X,\D)$ such that $\dim X=n$,  $\coeff(\D) \subset I$, and $K_X+\D$ is big. Then there exist a constant $\delta=\delta(n,I)$ and a positive integer $m=m(n,I)$ such that for any $(X,\D) \in \mathcal D$
 \begin{enumerate}
     \item $\vol(X,\D) \ge \delta$ and, more strongly, the set $\{\vol(X,D) \ | \ (X,\D) \in \mathcal D \}$ satisfies the DCC;
     \item (effective birationality) the linear system $|m(K_X+\D)|$ defines a birational map. 
\end{enumerate}
\end{theorem}

For families over curves we get the following relative version of \autoref{thm:HMX}.

\begin{corollary}\label{corinotro:relative_volume} [see \autoref{cor:relative_volume}]
Let $n$ be a positive integer and $\Lambda \subset \mathbb Q \cap [0,1]$ be a finite set. Then there exists a constant $\delta=\delta(n,\Lambda) >0$ such that 
$$
(K_{X/T}+\D)^{n+1} \geq \delta,
$$
for any fibration $f\colon (X,\Delta)\to T$ where $(X,\D)$ is a projective normal log-pair with general fibre klt and $T$ is a smooth curve such that $K_{X/T}+\D$ is $f$-ample and that $f$ has maximal variation.
\end{corollary}

 Such statement is strongly related to the theory of holomorphic foliations.  
 In recent years, there has been substantial interest in the birational aspects of foliations, especially the algebraically integrable ones, see for instance \cite{ACSS, CSMMP,CHLX}. When a foliation is defined by a regular fibration with reduced fibres, the divisor \( K_\mathcal{F} \) agrees with the relative canonical class of the fibration.
The analog of effective birationality of \autoref{thm:HMX} fails in general for foliations, as shown in \cite[Theorem 1.3]{Lu}, and further examples can be found in \cite{Passantino}. Nevertheless, the following question remains open even in dimension 2 and has motivated several recent investigations (\cite{PS19, SS23, LT24, HJLL}):
	
\begin{question}[J.V.\ Pereira,  \cite{Cas21}, \cite{HL21}]\label{q:volume}
Given positive integers $r$ and $n$, does it exist a positive constant $\varepsilon_{r,n}$ such that the volume of every $n$-dimensional rank $r$ canonical foliated varieties $(X, \mathcal F)$ of general type is at least $\varepsilon_{r,n}$?    
\end{question}
In the recent preprint \cite{HJLL}, the authors show that the volumes of log canonical algebraically integrable foliations belongs to a discrete set depending only on its rank and the volume of its general leaves. 

Our \autoref{corinotro:relative_volume} gives a positive answer to \autoref{q:volume} for a class of foliations induced by fibrations with reduced fibres. A positive answer to the following question, would allow to generalise such result to a large class of corank 1 foliations.

\begin{question}\label{q:positivity_foliations_*}
Let $(X,\Delta,\cF)$ be an lc (log canonical) foliated triple satisfying property (*) - see \cite{ACSS} for the definition - with $K_{\cF}+\Delta$ big and klt general leaf. Let $f\colon X\to B$ be the fibration inducing the foliation and assume that $\dim B=1$ and the relative dimension is $n$.  
Is $f_*\O_X(q(K_{\cF}+\Delta))$ ample for an effective $q \in \mathbb N$ depending only on $n$ and on the coefficients of $\D$?
\end{question}

\subsection{Automorphisms groups}
The classical Hurwitz's theorem says that the size of the automorphisms group of Riemann surface is at most 84 times its genus, i.e. a constant which depends only on the dimension times the volume. This result has been extended to higher dimension varieties in \cite[Theorem 2.8]{Ale} and \cite[Theorem 1.1]{HMX}. Using our bounds about volumes, we can further extend these results to the to the relative setting and to foliations.

Given a fibration $f\colon (X,\Delta)\to B$, we denote with $\Aut(f)$ the set of regular automorphisms $\tau$ of $(X,\D)$ such that $f \circ \tau = f$. 
\begin{corollary}[= \autoref{cor:autksb}]\label{thmintro:aut}
Let $n$ be a positive integer and $\Lambda \subset \mathbb Q \cap [0,1]$ be a DCC set. Then there exists a constant $\delta=\delta(n,\Lambda) >0$ such that 
$$
|\Aut(f)| \le \delta^d \vol(K_{X/B}+ \D)
$$
for any stable family $f\colon (X,\Delta)\to B$	satisfying the following conditions:  the coefficients of $\Delta$ are in $\Lambda$, the relative dimension of $f$ is $n$, $K_{X/T}+\Delta$ is $f$-ample  $f$ has maximal variaiton, at least one fiber is klt, and  $B$ is a smooth projective curve.

\end{corollary}

The analogous statement for foliations is open even in dimension 2, for related results, see \cite{PS02, CF14, CM19, SS23}.

\subsection{Consequences for the moduli spaces of stable pairs}

Let $n$ be a positive integer, $I$ a finite subset of $[0,1]\cap \Q$ closed under addition\footnote{A subset $I$ of $[0,1]\cap \Q$ is closed under addition if for every finite subset $\{x_i\}$ of $I$, if $\sum x_i\leq 1$, then $\sum x_i$ is in $I$.}, and $v$ a positive rational number. 

Following \cite{Kobook}, \cite[Section 6]{KP} and references therein, we consider the moduli stack $\M=\M_{n,v,I}$ of stable pairs of dimension $n$, volume $v$ and coefficients in $I$, endowed with its reduced structure. In particular, for every reduced scheme $B$, $\M(B)$ is the collection of stable families over $B$ of dimension $n$, volume $v$, and such that the coefficients of the boundary are in $I$. As explained in loc. cit. , the stack $\M$ exists, and it is a Deligne-Mumford stack with a proper coarse moduli space $M=M_{n,v,I}$. 

The projectivity of $M$ is one of the main result of \cite{KP,Fuj,PX}. In particular, in \cite{PX} the author shows that the Chow-Mumford line bundle $\lambda_{CM}$ is ample, see \autoref{S:CM} for the definition. The top self-intersection of $\lambda_{CM}$ is called the Chow-Mumford volume of $M$; over the complex numbers, it is related to the Weil-Petersson metric on $M$ and the K\"{a}hler-Einstein metric of the pairs parametrized by $M$, see \cite{Weil-Petersson} and references therein. 

\begin{corollary}[= \autoref{cor:volume_moduli}, a uniform lower bound on Chow-Mumford volumes]
 Let $n$ be a positive integer, $\Lambda \subset \mathbb Q \cap [0,1]$ be a DCC set closed under addition. Then there exists a constant $C=C(n,I) >0$ with the following property. For every irreducible component $M$ of the coarse moduli space of $n$-dimensional stable pairs with coefficients in $\Lambda$ such that at least one parametrized pair is klt, we have the inequality
    $$
    (\lambda_{CM})^{\dim M}\geq (vC)^{-\dim M}
    $$
 where $v$ is the volume of the pairs parametrized by $M$.
\end{corollary}

Very little is known about the Picard group of $M_{n,v,I}$, except for the case of curves $n=1$. On $\overline{M}_{g}$, $\lambda_1$ is nef and big and $\lambda_{CM} - t\lambda_1$ is nef if and only if $t \le 1$. 
In addition, $\lambda_1$ and $\lambda_{CM} - \lambda_1$ are semiample. Their associated maps have been studied, and several birational models of $\overline{M}_{g}$ are known by the so-called Hassett-Keel program (for a detailed account, see \cite{CTV1,CTV2} and references therein).
Recall also that on $\overline{M}_{g}$,
$$
\lambda_q= \binom{q}{2} \lambda_{CM} + \lambda_1.
$$ 

When a convenient base change result such ours \autoref{prop:base_change_stable_families} holds, it is possible to define the Hodge bundle $\rk(\E^{(q)})$, and the Hodge class $\lambda_q$ on the normalization $\nu \colon M^{\nu}\to M$, see \autoref{S:CM}. The following result was already known only for $q$ divisible enough, and in those cases $\lambda_q$ is already defined on $M$; our contribution is an effective condition on $q$. It summarizes \autoref{moduli}, \autoref{cor:semi_posit_lambda_e_CM} and \autoref{thm:postivity_lambda}.

\begin{corollary}[Effective positivity of lambda classes on moduli spaces]
 Let $I\subset [0,1]$ be a finite set closed under addition, $q$ a positive integer such that $qI\subset \mathbb{Z}$, and $M$ an irreducible component of the moduli space of stable pairs such that the coefficients of the boundaries are in $I$. Then $\lambda_q$ is a well-defined nef $\Q$-divisor on $M^{\nu}$. 

 Moreover, if $q\geq 2$, $\rk(\E^{(q)})>0$,  and at least one pair parametrized by $M$ is klt, then $\lambda_q$ is big.
\end{corollary}

For some consequence on the ample cone of $M_{n,v,I}$, see \autoref{thm:ample_cone}.

\subsection*{Acknowledgments} We thank F. Ambro, J. Cao, P. Cascini, S. Filipazzi, J. Liu, C. Spicer, R. Svaldi and F. Viviani for useful conversations. We thank X. L\"u for pointing out a mistake in a proof in the first version of this manuscript. 

GC and LT are partially supported by the GNSAGA group of INdAM. GC is also partially supported by the project MatMod@TOV and the grant ``Prin Moduli spaces and birational geometry
(2022L34E7W)".  LT is  partially supported by the PRIN2020 research grant ``2020KKWT53”.

\section{Notation and preliminary results}\label{S:notations}
Recall that our base-field $k$ is algebraically closed and of characteristic $0$.

\medskip 

\textbf{Varieties and pairs.} The term \emph{variety} means an equidimensional reduced separated connected scheme of finite type over $k$. A variety can have multiple irreducible components. An open set is called \emph{big} if its complement has codimension at least two. 
A variety is \emph{demi-normal} if it is S2, and codimension 1 points are either regular points or nodes, see \cite[Definition 5.1]{Kosing}. Following \cite[Definition 10.31]{Kobook}, the volume of a Mumford divisor $D$ is $\lim_{m\to \infty}\frac{h^0(X,\O_X(mD))}{m^{\dim X}/\dim X!}$, and for a nef divisor it is equal to $D^{\dim X}$.  For the benefit of the reader, in some statements we will stress the demi-normality.

A \emph{pair} $(X,\Delta)$  consists of a demi-normal variety $X$ and an effective $\Q$-Mumford divisor $\Delta$, where a \emph{Mumford divisor} is a Weil divisor with no  irreducible component of its support being contained in the singular locus of $X$. We don't assume in general that $K_X + \D$ is $\Q$-Cartier, this is relevant in some situation (for instance \autoref{lem:snc_Cartier}). 
Following \cite[Definition 10.31]{Kobook}, the volume of a Mumford divisor $D$ is $\lim_{m\to \infty}\frac{h^0(X,\O_X(mD))}{m^{\dim X}/\dim X!}$, and for a nef divisor it is equal to $D^{\dim X}$.

\medskip

\textbf{Fibrations and families.} A proper surjective morphism of varieties $f: X \to B$ is called a \emph{fibration} if $f_*\O_X = \O_B$ and it is called a \emph{family} if it is a fibration and flat. When we write $f: (X, \Delta) \to B$ is a fibration, then we mean that $f : X \to B$ is a fibration and $(X, \Delta)$ is a pair.

For a fibration where both $X$ and $B$ are $S_2$, the relative canonical sheaf is a divisorial sheaf. We can and do assume that none of the irreducible components of the support of the canonical divisor $K_X$ and the relative canonical divisor $K_{X/B}$ are contained in the singular locus of $X$. This is possible if we choose the divisor associated to a general rational section of the canonical or relative canonical sheaf.

\medskip

\textbf{Locally stable families.} A family $f\colon (X,\Delta)\to B$ over a variety $B$ is \emph{locally stable} if the fibers of $f$ are demi-normal, $\Delta$ is an effective $\Q$-divisor, $\Supp \Delta$ does not contain any fiber, and singular codimension $1$ point (i.e., the node) of any fiber, $K_{X/B}+ \Delta$ is $\Q$-Cartier, and $\big(K_{X_b}, \Delta_b \big)$ is slc for every $b \in B$. It is \emph{stable} if in addition $K_{X/B}+ \Delta$ is $f$-ample.
A locally stable family over a smooth base such that the general fiber is klt has automatically klt total space (see \cite[Proposition 2.15]{Kobook} when the base is a curve, and \cite[Corollary 4.56]{Kobook} for a general base).

\subsection{Fiber product when the total space is normal}\label{S:fiber_product}

Let $f:X \to T$ be a family where $T$ is a smooth projective curve. Given an integer $r$, we denote by
$$
f^{(r)}\colon X^{(r)}\to T
$$
the $r$-th fiber self-product of $X$ over $T$, and denote by $p_i$ the projection onto the $i$-th factor. If both $X$ and $X^{(r)}$ are normal, and $D$ is a Weil divisor on $X$, then we define
\begin{equation*}
D^{(r)} = \sum_{i=1}^p p_i^* D.
\end{equation*}

\begin{lemma}{\cite[Lemma 6.3]{CP}}\label{lem:normal}
Assume that  $X$ is normal and all fibers of $f$ are reduced, then
\begin{enumerate}
    \item $X^{(r)}$ is normal for every $r\geq 1$;
    \item every base change $X_S$ is normal, where $S$ is a smooth projective curve and the morphism $S\to T$ is finite.
\end{enumerate} 

\end{lemma}

In the situation of Lemma \ref{lem:normal}, we have $K_{X^{(r)}/T}=K_{X/T}^{(r)}$, see \cite[Prop 2.26]{Pat14}.  Moreover, in this situation, if $(X,\Delta)$ is a pair,  then so is $(X^{(r)},\Delta^{(r)})$.

\begin{lemma}[K\"{u}nneth formula for divisorial sheaves]\label{lem:kunneth}
Let $f: X \to T$ be a morphism of finite type to a smooth projective curve such that the general fibers are normal and all fibers of $f$ are reduced.  Let $D$ be a $\Z$-divisor on $X$. Then for any integer $n \geq 1$ we have
$$f^{(n)}_* \O_{X^{(n)}} \left(D^{(n)} \right) \cong \big( f_* \O_X(D) \big)^{\otimes n}$$
\end{lemma}

\begin{proof}
First, observe that  by  \autoref{lem:normal}, $X$ is normal and the $X^{(n)}$ are normal.  Let $V \subseteq T$ be a non-empty open set, we want to show that $f^{(n)}_* \O_{X^{(n)}} \left(D^{(n)} \right)(V) \cong \big( f_* \O_X(D) \big)^{\otimes n}(V)$, and the isomorphism is compatible with inclusions. Let $U \subseteq X$ be locus where $f$ is smooth.  By our assumptions $U^{(n)} \cap f^{(n), -1} V$ is a big open set of $f^{(n), -1} V$. Hence, the global sections of $D^{(n)}$ over $f^{(n), -1} V$ and over $U^{(n)} \cap f^{(n), -1} V$ agree. So, by replacing $X$ by $U$ we may assume that $X$ is regular and hence $D$ is Cartier (note, we did not assume that $f$ was projective). Then, we denote $\O_X(D)$ by the line bundle $\cL$ and the ``usual proof'' works. I.e. we show that for every $n>1$ we have $f_*^{(n)} \O_X(D^{(n)}) \cong \Big( f_* \O_X(D) \Big) \otimes f_*^{(n-1)} \O_X(D^{(n-1)}) $ via the following computation,
\begin{multline*}
f^{(n)}_* \cL^{(n)} \cong f^{(n-1)}_* g_* (g^* \cL^{(n-1)} \otimes \pr^* \cL) 
\expl{\cong}{projection formula}
f^{(n-1)}_* \left( \cL^{(n-1)}  \otimes g_* (  \pr^* \cL) \right)
\\ \expl{\cong}{flat base-change}
f^{(n-1)}_* \left( \cL^{(n-1)}  \otimes f^{(n-1),*} f_* \cL \right)
\expl{\cong}{projection formula}
\big( f_* \cL \big) \otimes f^{(n-1)}_* \left( \cL^{(n-1)}    \right)
\end{multline*}
where $g : X^{(n)} \to X^{(n-1)}$ is the projection on the first $n-1$ factor, and $\pr : X^{(n)} \to X$ is the projection on the last factor. The key in this argument is that both the projection formula and flat base-change works for any finite type morphism, no need to assume projectivity or properness, see \cite[Tags 01E8 and 02KH]{stackproject}. 
    
\end{proof}


\subsection{A slope inequality}\label{S:slope}
In this paper, we will use a special case of \cite[Theorem F(1)]{CTV} multiple times; for the reader convenience, we recall here the statement that we will need. With respect to loc. cit. , we take $q=1$, and we assume that the map induced by $L$ when restricted to a fiber is birational, which is stronger than generically finite.

\begin{theorem}{\cite[Theorem F(1)]{CTV}}\label{slopeinequality}
Let $f\colon X\to T$ be fibration, where $X$ is a normal projective variety of dimension $n+1$, and $T$ is a smooth projective curve. Let $L$ be an $f$-ample $\Q$-Cartier $\Z$-divisor on $X$ such that $L$ and $f_*\O_X(L)$ are nef, and that $L$ restricted to a fiber of $f$ gives a birational map. Then the following inequality holds:
$$
L^{n+1}\geq \deg f_*\O_X(L) \,.
$$
\end{theorem}

Let us take the opportunity to formulate the following natural question; a positive answer could for instance strengthen \autoref{thm:volumelambda}, obtaining, with the notations of \autoref{thm:volumelambda}, lower bounds on $(K_{X/B}+\Delta)^{n+d}$ with $d>1$.

\begin{question}[Higher dimensional slope inequality]\label{question}
In the set-up of \autoref{slopeinequality}, if we replace the curve $T$ with an higher dimensional projective manifold, does some convenient generalization of \autoref{slopeinequality} hold?
\end{question}

\section{Chow-Mumford line bundle and lambda classes}\label{S:CM}
We define some bundles and divisors on the base of any fibration; under suitable assumptions, they will satisfy convenient base change result to give objects on moduli spaces.

\begin{definition} \label{S:def}
For a fibration $f : (X, \Delta) \to B$ over a reduced base, and for any integer $q>0$ we define \begin{enumerate}
     \item  the $q$-th Hodge bundle as $f_*\O_X \Big(q(K_{X/B}+\Delta)\Big)$, we denote it by $\E^{\Delta}_{q,f}$;
     
     \item the $q$-th lambda class $\lambda_{q,f}^{\Delta}$ is the Cartier divisor $\det\left(\E^{\Delta}_{q,f}\right)$;
     
     \item when $B$ is normal, the CM line bundle $\lambda_{CM,f}^{\Delta}$ is the cycle push-forward $f_*(K_{X/B}+\Delta)^{n+1}$, where $n=\dim X- \dim B$.  It is known that $\lambda_{CM,f}^{\Delta}$ is $\Q$-Cartier, for example by \cite[Section 2.3]{PX} with further details in \cite[Sections 2.4, 3 and Appendix]{CP}.

 \end{enumerate}
 When no risk of confusion arises, we omit the dependence from the map $f$ or the boundary $\Delta$ in the notations.
\end{definition}

\begin{proposition}[Approximation of the Chow-Mumford line bundle with lambda classes]\label{prop:lambda_and_CM}
With the above notations, let $f\colon (X,\Delta)\to B$ be a stable family over a normal quasi-projective variety $B$, then there exists a positive integer multiple $\ell$ of the Cartier index of $K_{X/B}+\Delta$ such that 
$$
\lambda_{CM,f}^{\Delta}=(n+1)!\lim_{k\to \infty}(\ell k)^{-n-1}\lambda_{\ell k,f}^{\Delta}\,.
$$

 \end{proposition}
 \begin{proof}
When $X$ is normal, the result follows from the Knudsen-Mumford expansion (see e.g. \cite[Lemma A.2 (a)]{CP}). (In loc. cit. the base can be quasi-projective instead than projective; the base is also assumed to be smooth, however, as already argued in the proof \cite[Lemma A.1]{CP}, one can deduce the statement about normal base restricting to the regular locus, which is a big open subset of $B$.) We can reduce the demi-normal case to the normal case using \cite[Lemma A.4]{CP} (The proof of loc. cit. works for any normal base, not only a curve.)

 \end{proof}

\subsection{Base change for stable families}

The following result in a colloquial language means that the Hodge bundle, the Chow-Mumford and the lambda classes satisfy base-change for stable families, or that their formation commutes with base-change. It is well known for the CM line bundle, and for the Hodge bundle when $q$ is divisible enough; the novelty is that we give an effective condition on $q$ when the base is normal. It also shows that the Hodge bundles are indeed vector bundles under the assumption that $q\D$ is a $\mathbb Z$-divisor.

\begin{proposition}[Base change for stable families]\label{prop:base_change_stable_families}
Let $B$ be a normal quasi-projective variety and $f\colon (X, \Delta)\to B$ a stable family. Let $g\colon T\to B$ be a morphism from a normal quasi-projective variety. Let $h\colon (X_T,\Delta_T)\to T$ be the family obtained after base change via $g$ (see 
the proof for the definition of $\Delta_T$). Then $h$ is a stable family, and if $q\Delta$ is a $\Z$-divisor, then we have
$$g^*\E_{q,f}^{\Delta} \cong \E_{q,h}^{\Delta_T} \; , \qquad g^*\lambda_{q,f}^{\Delta}=\lambda_{q,h}^{\Delta_T}\quad \textrm{and} \quad g^*\lambda_{CM,f}^{\Delta}=\lambda_{CM,h}^{\Delta_T} \; ,$$

the sheaf $\E_{q,f}^{\Delta}$ is locally free, and $\lambda_{q,f}^{\Delta}$ is Cartier. 

If $q$ is divisible by the Cartier index of $K_{X/B}+\Delta$, the above base change for $\E_{q,f}^{\Delta}$ and  $\lambda_{q,f}^{\Delta}$ holds also if $B$ and $T$ are just reduced rather than normal.
\end{proposition}

\begin{proof}
Let $G$ be the natural morphism $G\colon (X_T,\Delta_T)\to (X,\Delta)$.
 Let $U \subseteq X$ be an open set, such that:
 \begin{itemize}
     \item $\codim_{X_b} X_b \setminus U_b \geq 2$ for every $b \in B$, 
     \item $U$ is relatively Gorenstein over $B$, and 
     \item $\Supp \Delta|_U$ is contained in the regular locus of $U$.
 \end{itemize}
 Such open set $U$ exists by the stable assumption (see \cite[Definition 3.35]{Kobook} for the first two conditions, while the third follows from the fact that $\D$ is a Mumford divisor, see \cite[Paragraph 4.2.4]{Kobook}). By the relative Gorenstein and the regularity assumptions we have the following equality of $\Q$-divisors
 \begin{equation}
 \label{eq:bc_on_U}
     G_U^* (K_{U/B} + \Delta|_U) = K_{U_T/T} + \Delta_{U_T},
 \end{equation}
 where $G_U : U_T \to U$ is the restriction of $G$.
 Then, using that $U$ is big in all fibers, we obtain
\begin{equation}\label{eq:bc_can}
G^*(K_{X/B}+\Delta)=K_{X_T/T}+\Delta_T \,,
\end{equation}
where $\Delta_T$ is the unique extension of $\Delta_{U_T}$ (via the closure).
Then, the base change $h\colon (X_T,\Delta_T)\to T$ is stable because of Equation \autoref{eq:bc_can} and \cite[Definition-Theorem 4.7.1]{Kobook}.

We want to show that
\begin{equation}
\label{eq:log_canonical_pullback}
    G^* \O_X\Big(q(K_{X/B} + \Delta)\Big) \cong \O\Big(q(K_{X_T} + \Delta_T)\Big).
\end{equation}
and that the above sheaves are flat over $B$ and $T$, respectively. If $q(K_{X/T}+\Delta)$ is Cartier this standard. Otherwise we have to assume that $B$ and $T$ are normal make the following (lengthy) argument (we speculate that there exists an argument that does not use the normality assumption).

In what follows we are going to use that $X$ is demi-normal, for instance to have that all divisorial sheaves are reflexive. To see that $X$ is demi-normal note first that it is $S_2$ since $B$ and all fibres are $S_2$ (cf. \cite[Paragraph 10.10]{Kobook}). 
Consider then a singular an irreducible component $D$ of the singular locus of $X$ which has codimension 1. Since $B$ is smooth in codimension one, there is a dense open subset $U$ of $D$ such that $f(U)$ is in the smooth locus of $B$. Since the fibers of the map $f\colon U \to f(U)$ are nodal in codimension one, we conclude that the generic point of $D$ is a node. 

 Let $\iota : U \hookrightarrow X$ and $\iota_T: U_T \to X_T$ be the natural embeddings. Then for every integer $s$ we have
\begin{equation*}
\O_{X}\Big(s(K_{X/B}+ \Delta)\Big)
\expl{\cong}{the left side is a divisorial sheaf and hence reflexive} 
\iota_* \O_{U}\Big(s(K_{U/B}+ \Delta|_U)\Big) 
\end{equation*}
and
\begin{equation*}
\O_{X_T}\Big(s(K_{X_T/T}+ \Delta_T)\Big)
\expl{\cong}{the left side is a divisorial sheaf and hence reflexive} 
\iota_{T,*} \O_{U_T}\Big(s(K_{U_T/T}+ \Delta|_{U_T})\Big)
\expl{\cong}{by \autoref{eq:bc_on_U}} 
\iota_{T,*} G_U^* \O_{U}\Big(s(K_{U/B}+ \Delta|_U)\Big) 
\end{equation*}
So, the question is when do we have $G^* \iota_* \cL \cong \iota_{T,*} G_U^* \cL$ for an adequate line bundle $\cL$ on $U$ (we would apply this to $\cL=\O_{U}\Big(s(K_{U/B}+ \Delta|_U)\Big)$). According to \cite[Prop 16]{Kol18}, if we combine this isomorphism with the flatness of the corresponding sheaves, for the two conditions together to hold (for every $T$ as above) we may assume that $B$ is the spectrum of a DVR and $T$ the closed point of $B$. But, then $\O_{X_T}\Big(s(K_{X_T/B}+ \Delta_T)\Big)$ is the reflexive hull of $G_U^* \O_{U}\Big(s(K_{U/B}+ \Delta|_U)\Big)$. Hence, $G^* \O_{X}\Big(s(K_{X/B}+ \Delta)\Big)|_T$ is isomorphic to it if and only if $\O_{X}\Big(s(K_{X/B}+ \Delta)\Big)$ satisfies the depth condition defining $S_3$ for every point of $T$ (here we used that depth decreases exactly by $1$ when restricting over a Cartier divisor). This condition we call being $S_3$ along $T$. Additionally, for flatness one has to show an even milder condition:  all sections of $\O_X\Big(s (K_{X/B}+ \Delta)\Big)_T$ are supported at generic points \cite[Lem 2.13]{BHPS13}, for which it is enough to show that $\O_X\Big(s (K_{X/B} + \Delta)\Big)$ is $S_2$ along $T$. 

However, the points of $X_T$ are not log canonical centers of $(X, \Delta)$, as even $(X , \Delta + X_T)$ is slc by inversion of adjunction. Hence, it is enough to show that the $S_3$ condition holds at all non-log canonical centers of $(X, \Delta)$ for the sheaf $\O_X\Big(s (K_{X/B} + \Delta)\Big)$. This is not true definitely without requiring some further conditions on $s$ or on $\Delta$. However, for $q=s$
one can show the above depth condition at non log canonical centers by applying \cite[Thm 3.(i)]{Kol11} directly with $\Delta'=0$ and $D=q(K_{X/B} + \Delta)$, as $D=q(K_{X/B}+ \Delta)$ is $\Q$-Cartier. This conclude the proof of Equation \autoref{eq:log_canonical_pullback}.

Equation \autoref{eq:log_canonical_pullback} together with the flatness of the participating sheaves yields a natural morphism $\rho\colon g^*\E_{q,f}^{\Delta}\to \E_{q,h}^{\Delta_T}$. 
When $q=1$, $\rho$ is an isomorphism because of \cite[Theorems 8.16]{Kobook} (applied with $L$ trivial and $B = \Delta$). 
To show that it is an isomorphism for $q\geq 2$, it is enough to see that 
\begin{equation}
\label{eq:vanishing}
 H^i\bigg(X_b,\O_{X_b}\Big(q(K_{X_b} + \Delta_b)\Big)\bigg)=0 \textrm{,  for } i>0 \textrm{  and for every closed point } b \in B. 
\end{equation}
Indeed, as we already know that $\O_X\Big(q(K_{X/B}+ \Delta)\Big)$ is flat over $B$ and that 
$$\O_X\Big(q(K_{X/B}+ \Delta)\Big)|_{X_b} \cong \O_{X_b}\Big(q(K_{X_b} + \Delta_b)\Big),$$
the standard cohomology and base change, e.g. \cite[Corollary 2 page 50]{MumAbel} gives that $\rho$ is an isomorphism in this case. However,  \autoref{eq:vanishing} follows by applying \cite[Theorem 1.7]{Fuj14} with $D=q(K_{X/B}+\Delta)$. 

By the same token, the vanishing \eqref{eq:vanishing} toghether with \cite[Corollary 1 (ii) and Corollary 2 page 50]{MumAbel} implies that $\E_{q,f}^{\Delta}$ is locally free. $\lambda_{q,f}^{\Delta}$ is Cartier because it is the determinant of a locally free sheaf.

Taking the determinant of $\rho$, we obtain the isomorphism $g^*\lambda_{q,f}=\lambda_{q,h}$. The base change for $\lambda_{CM}$ now follows from \autoref{prop:lambda_and_CM} if the base is normal.

\end{proof}
\begin{corollary}\label{moduli}
Let $\M$ be an irreducible component of a moduli stack parameterizing stable pairs. We endow $\M$ with its reduced structure.

Let $A$ be the least common multiple of the cardinality of $\Aut(F,\Gamma)$, where $(F,\Gamma)$ varies in $\M(\Spec(k))$. Then:

\begin{enumerate}
    \item 
the assignment of the vector bundle $\E_{q,f}^{\Delta}$, and the line bundle $\lambda_{q,f}^{\Delta}$ from \autoref{S:def}, yields for all positive integer $q$ divisible enough respectively a vector bundle $\E_q$, and a line bundle $\lambda_q$ on $\M$.

\item  The assignment of the vector bundle $\E_{q,f}^{\Delta}$, the line bundle $\lambda_{q,f}^{\Delta}$ and the $\Q$-divisor $\lambda_{CM,f}^{\Delta}$  from \autoref{S:def} for families over normal basis yelds, for all positive integers $q$ such that $qI\subset \Z$, a vector bundle $\E_q$, a line bundle $\lambda_q$, and a $\Q$-divisor $\lambda_{CM}$ on the normalization $\M^{\nu}$ of $\M$.  

Moreover, $\lambda_q^{\otimes A}$ descends to a Cartier divisor on the coarse moduli space $M^{\nu}$ of $\M^{\nu}$.

\item There exists a  $\Q$-divisor on $\M$, which by abuse of notations we still denote by $\lambda_{CM}$, which pull-back to our $\lambda_{CM}$ via $\nu$.
\end{enumerate} 
\end{corollary}
\begin{proof}
If we present a DM stack as a groupoid quotient $[R/U]$, to define a bundle on the stack is equivalent to give a bundle $L$ on $U$ and an isomorphism on $R$ between $s^*L$ and $t^*L$, where $s$ and $t$ are the source and target map for the relations.

We can find a presentation of $\M$
 with $R$ and $U$ reduced, and a presentation of $\M^{\nu}$ wit $R$ and $U$ normal (see e.g. \cite[0GMH]{stackproject}). We can thus define the convenient bundles on $U$ and use the base change from \autoref{prop:base_change_stable_families} to obtain the requested isomorphisms.

A line bundle $L$ on $\M^{\nu}$ descends to a Cartier divisor on $M^{\nu}$ if for every pair $(F,\Gamma)$ in $\M$ the action of $G=\Aut(F,\Gamma)$ on the fiber $L_{(F,\Gamma)}$ is trivial. Since $L_{(F,\Gamma)}$ is one dimensional and the cardinality of $G$ divides $A$, the action of $G$ on $L_{(F,\Gamma)}^{\otimes A}$ is trivial.

The construction of $\lambda_{CM}$ on $\M$, which in turn is the construction of $\lambda_{CM}$ on the base of every stable family, is carried out using the Knudsen-Mumford expansion, see \cite[Section 2.3]{PX} and \cite[Section 3]{CP}. As explained in loc. cit., over normal basis, and so over $\M^{\nu}$, it coincied with our definition.

\end{proof}

As shown by the following example, there is no hope to have a base-change as in \autoref{prop:base_change_stable_families} without assuming something on $q$ and $\Delta$ (as here we assumed that $q \Delta$ is a $\Z$-divisor). However, one might wonder if it works for all integers $q$ in the case of diminished standard coefficients, or even in the case of coefficients greater than $1/2$. This is in fact almost shown in \cite{Kol18}, except that it is assumed that general fibers are normal. The case of  non-normal  general fibers is still open. 

\begin{example}\label{ex:basechange}
     
 Let us consider a conic $C \subseteq \P^1 \times \P^1$ such that $C \cap \pr_2(0)$ is a double point and another general such conic $C'$, where $\pr_2 : \P^1 \times \P^1 \to \P^1$ is the projection on the second factor. Let $D$ be some general high enough degree curve on $\P^1 \times \P^1$. Set $X= \P^1 \times \P^1 \times \P^1$ and $\Delta= \frac{1}{2}( \P^1 \times C + \P^1 \times C'+\P^1 \times D$). Let $f : X \to B= \P^1 \times \P^1$ be the projection on the first and third factor and let $g : T = \P^1 \to \P^1 \times \P^1$ be the map $x \mapsto (x, 0)$. Then, $\E_{1,f}^{\Delta}=0$ but $\E_{1,h}^{\Delta_T} \neq 0$. 
  \end{example}

\subsection{Flat base change}

\begin{proposition}[Flat base change]\label{prop:flat_base_change}
Let $f:(X,\D) \to B$ be a fibration, where $(X,\D)$ is a log pair and $B$ is a normal quasi-projective variety.
Let $g\colon B'\to B$ be a flat morphism, and let $\Sigma\subset B$ the branched divisor. Assume that there exists a neighbourhood $U$ of $\Sigma$ such that $f$ restricted to $f^{-1}U$ is locally stable. Let $h\colon (X_{B'},\Delta_{B'})\to {B'}$ be the family obtained after base change. Then
\begin{enumerate}
    \item  if $(X,\Delta)$ is slc/lc/klt, the same is true for $(X_{B'},\Delta_{B'})$;
    \item if the generic fiber of $f$ is slc/lc/klt, the same is true for the generic fiber of $h$;
    \item $g^*\E_{q,f}^{\Delta}=\E_{q,h}^{\Delta_{B'}} \, , \, g^*\lambda_{q,f}=\lambda_{q,h}\, \textrm{and } g^*\lambda_{CM,f}=\lambda_{CM,g}$ for any integer $q \ge 1$.
\end{enumerate}
\end{proposition}
\begin{proof}
Let $G\colon (X_{B'},\Delta_{B'})\to (X,\Delta)$ be the induced flat morphism. First observe that $G^*(K_{X/B}+\D)=K_{X/{B'}}+\D_{B'}$ because: on the complement of $f^{-1}U$, $G$ is \'etale, so $G^*(\omega_X)=\omega_{X_{B'}}$, which implies $G^*(K_{X/B}+\D)=K_{X/{B'}}+\D_{B'}$ by definition of $\D_{B'}$; on $f^{-1}U$, it is true because the family is locally stable, see \cite[Theorem 2.67]{Kobook}. This already gives Items (1) and (2).

Since flat pull-back of a reflexive sheaf is a reflexive sheaf, we
also have $G^*\O_X(q(K_{X/B}+\Delta))=\O_{X_{B'}}(q(K_{X_{B'}/{B'}}+\Delta_{B'}))$ for any $q \in \mathbb N$. 
This gives a map 
$$
\rho \colon g^*f_*\O_X(q(K_{X/B}+\Delta)\to h_*\O_{X_{B'}}(q(K_{X_{B'}/{B'}}+\Delta_{B'}))
$$
for any $q \in \mathbb N$, which is an isomorphism by flat base change in cohomolgy. Taking the determinant of $\rho$ we obtain the statement for $\lambda_q$. 

The base change for $\lambda_{CM}$ follows from flat base change for Chow groups.
\end{proof}

\subsection{Stable reduction and Hodge bundles}\label{S:base_change_non_reduced_fibers}

Stable reduction is used to transform a family over a curve where the generic member is stable into a stable family. It is useful to get rid of non-reduced fibers. In this section we investigate the interplay between stable reduction and positivity of the Hodge bundle.

The following is a combination of \cite[Proposition 2.1 and Lemma 6.2]{CP}.

\begin{lemma} 
\label{lem:fibre_ridotte}
Let $(X,\Delta)$ be a normal pair, $f: X \to T$ be a surjective morphism, where $T$ is a smooth projective curve.  Let $\tau : S \to T$ be a finite surjection of smooth projective curves, $\pi: Y \to X_S:= X \times_T S$ the normalisation and $g : Y \to S, h: X_S \to S, \rho : X_S \to X$ the induced morphisms. The notations are summarized by the following diagram
\begin{equation}
\label{eq:base_change}
\xymatrix@C=70pt{
Y \ar[r]^{\pi}  \ar[dr]_g  &  
X_S \ar[r]^{\rho} \ar[d]^h & X \ar[d]^f \\
& S \ar[r]_{\tau} & T
}
\end{equation}
Then
\begin{enumerate}
\item\label{existence_cover} there exists a choice of $\tau$ such that all fibers of $g$ are reduced (any $\tau$ whose degree is a common multiple of the multiplicity of every fiber of $f$ work).
\item \label{item_boundary} for any choice of $\tau$ additionally there exists a boundary $\Gamma$ on $Y$ such that
\begin{enumerate}
\item\label{item_crepant}  $(\rho \circ \pi)^*(K_{X/T}+\Delta)=K_{Y/S}+\Gamma$;
\item\label{item_sing} if the generic fiber of $f$ is klt (resp. lc), then the generic fiber of $g$ is klt (resp. lc);

\item\label{item_eff} there exists an effective vertical Weil divisor $E$ on $Y$ such that $(\rho\circ \pi)^*\Delta=\Gamma-E$ and  $(\rho\circ \pi)^*K_{X/T}=K_{Y/S}+E$;
\item \label{item_positivity} if $K_{X/T}+\Delta$ is $f$-ample (resp. f-big, nef), then $K_{Y/S}+\Gamma$ is $g$-ample (resp. g-big, nef).
\item\label{item_CM} We have the base change $\tau^*\lambda_{CM}^{f}=\lambda_{CM}^g$.
\end{enumerate}
\end{enumerate}
\end{lemma}

\begin{proof}
Item \autoref{existence_cover} is \cite[Lemma 2.53]{Kobook}.

To prove Item \autoref{item_boundary} we follow \cite[Subsection 2.4.2]{CP}, where \cite[Proposition 2.1]{CP} is proven. The existence of an effective boundary $\Gamma$ such that $(\rho \circ \pi)^*(K_{X/T}+\Delta)=K_{Y/S}+\Gamma$ is explained in loc. cit.. This gives Item \autoref{item_crepant}.

Item \autoref{item_sing} is true because the generic fibers are isomorphic. Item \autoref{item_positivity} holds true because $\rho \circ \pi$ is finite.

We now  prove Item \autoref{item_eff}. Set $E=(\rho\circ \pi)^*K_{X/T}-K_{Y/S}$. Since $\rho\circ \pi$ is \'{e}tale over an open subset of $T$, $E$ is vertical. We have to show that $E$ is effective. Let $U \subset X$ be the open subset of relatively Gorenstein points over $T$. In \cite[Subsection 2.4.2]{CP} it is proven that $\rho^{-1}U$ is big in $X_S$, and  that $\omega_{W/S}$ injects into $(\rho\circ \pi)^*\omega_{U/T}$, where $W= (\rho\circ \pi)^{-1}(U)$.  Since $W$ is big in $Y$, this implies that $E=(\rho\circ \pi)^*K_{X/T}-K_{Y/S}$ is an effective Weil divisor.

Item \autoref{item_CM} is as in  \cite[Proposition 3.7 and 3.8, see also Lemma A.4]{CP}.
\end{proof}

Let us stress that the base change from \autoref{item_CM} of \autoref{lem:fibre_ridotte} does not work if we replace the Chow-Mumford line bundle with $\lambda_q$ or the Hodge bundle. The defect can be measured by sheaves such as $(\O_Z/\O_{X_S})\otimes\omega_{(X/S)}^{[q]}$ from the proof of \cite[Lemma A.4]{CP}.

\begin{lemma}\label{lem:reduction_bundle}
In the situation of Lemma \ref{lem:fibre_ridotte},  let $L$ be an integral Weil  divisor on $X$. Then  $L_Y:=(\rho \circ \pi)^*( L- K_{X/T}) + K_{Y/S}$ is an integral Weil  divisor such that there exists a generically surjective homomorphism
\[
g_*\O_Y(L_Y)  \to   \tau^*f_*\O_X(L).
\]
In particular, if $g_*\O_Y(L_Y)$ is nef/ample, the same is true for $f_*\O_X(L)$.
\end{lemma}

\begin{proof}
Let $U \subseteq X$ be the intersection of the relative Gorenstein locus and of the locus where $L$ is Cartier. This  is a  big open set of $X$, as it contains the  regular locus of $X$. Set $U_S:= \rho^{-1} U$, and $U_Y= \pi^{-1} \rho^{-1} U$. Note that $X_S$ is $G_1$ and $S_2$, and hence the equivalence between reflexive and $S_2$ sheaves holds on $X_S$ as well \cite[Proposition 1.9]{HarGenDiv}, and also the usual extension property of $S_2$ sheaves hold \cite[Proposition 1.11]{HarGenDiv}. Using $[\otimes]$ to denote the reflexive tensor product ( i.e., taking tensor product and then reflexive hull), we have
\begin{multline*}
g_*\O_Y(L_Y) 
\expl{=}{The definition of $L_Y$} 
g_* \O_{Y} ((\rho \circ \pi)^*( L- K_{X/T}) + K_{Y/S}) 
\\
\explpar{150pt}{=}{flat pullback of reflexive is reflexive \cite[Proposition 1.8]{Har}, and hence $\rho^* \O_X(L -K_{X/T})$ is reflexive }
h_* \pi_*   (\pi ^{[*]} \rho^* \O_X(L - K_{X/T}) [\otimes]   \omega_{Y/S}) 
\explparshift{150pt}{-40pt}{=}{$\pi_*  \O_{Y} (\pi ^{[*]} \rho^* \O_X(L - K_{X/T}) [\otimes]   \omega_{Y/S}) $ is $S_2$, because it is a pushforward of an $S_2$ sheaf \cite[Prop 5.4]{KM}, so the next isomorphism can be checked on $U_S$}
\\
= h_* (\rho^* \O_X(L - K_{X/T}) [\otimes] \pi_*  \omega_{Y/S}) 
\explparshift{160pt}{-160pt}{\to}{This is $h_*(\_)$ applied to the reflexive twist of the trace. In particular, this map is isomorphism over the non-empty open set of $S$ over which $Y \to X_S$ is an isomorphism.} 
h_* (\rho^* \O_X(L - K_{X/T}) [\otimes]  \omega_{X_S/S}) 
\\
\explparshift{190pt}{-85pt}{=}{the isomorphism $\rho^* \O_X(L - K_{X/T}) [\otimes]  \omega_{X_S/S} \cong \rho^* \O_{X} ( L)$ holds on $U_S$,  by e.g. \cite[Theorem 3.6.1 page 164]{ConradBook}; we then extend it over all $X$ because both sides are reflexive; then we apply $h_*(\_)$}  
h_* \rho^* \O_{X} ( L)
\explshift{60pt}{=}{flat base-change} \tau^* f_* \O_X (L),
\end{multline*}
The last assertion follows from \autoref{lem:fuj2.2}.
\end{proof}

With the following proposition we recall the stable reduction, and we discuss how the Hodge bundles, the lambda classes and the Chow-Mumford line bundle are affected by this operation. Let us stress that the stable reduction does not preserve the relative log canonical bundle. We follow \cite[Theorem 2.51]{Kobook}. 

\begin{proposition}\label{prop:stable_reduction}
Let $f:(X,\D) \to T$ be a family where $(X,\D)$ is a projective normal pair with $K_X +\Delta$ $\Q$-Cartier and with $T$ a smooth curve. Assume that over an open dense subset $T^0 \subset T$, $f^0: (X^0,\D^0) \to T^0$ is stable, in particular all fibres are slc, $K_{X^0/T^0}+\Delta^0$ is $f^0$-ample and the support of $\Delta^0$ does not contain any vertical component. 

Then there exists a finite surjection $\tau: S \to T$ from a smooth projective curve $S$ such that the base change
$$
f^0_S: (X^0\times_T S, \D^0 \times_T S) \to \tau^{-1}(T^0)
$$ 
extends to a stable family $g: (Y,\Theta) \to S$, and for any $q \in \Z_{>0}$ such that $q \Delta$ is a $\Z$-divisor we have:
\begin{enumerate}
\item  $(Y,\Theta)$ is lc and $q \Theta$ is a $\Z$-divisor;
\item\label{prop:stable_reduction_itemHodge}   if $\mathcal E^{\Theta}_{g,q}$ is nef/ample, then the same is true for $\mathcal E^{\D}_{f,q}$;
\item\label{prop:stable_reduction_itemLambda} 
if $ \lambda^{\Theta}_{g,q}$ is nef/ample, then the same is true for $\lambda^{\D}_{f,q}$, and
\item\label{prop:stable_reduction_itemCM} assuming that $K_{X/T}+\Delta$ is $f$-ample, if $\lambda^{\Theta}_{CM,g}$ is nef/ample, then the same is true for $\lambda^{\D}_{CM,f}$. 

\end{enumerate}
\end{proposition}

\begin{proof}

By \cite[Corollary 10.46]{Kosing} we can take a log resolution $\phi: X' \to X$ such that ${\phi_{*}^{-1}}\D + \Exc(\phi) + (X'_t)_{\red}$ is simple normal crossing for every $t \in T$, where $\Exc(\phi)$  is the  exceptional locus of $\phi$ as a reduced divisor. We set $f' = f \circ \phi$. We define a boundary divisor $\D'$ on $X'$ taking the horizontal part (over $T$) of ${\phi_{*}^{-1}}\D + \Exc(\phi)$.
This way $q\D'$ is a $\Z$-divisor and $(X',\D' + (X'_t)_{\red})$ is lc for any $t \in T$. We also have 
\begin{multline}
\label{eq:stab_red_step_1}
K_{X'}+{\phi_{*}^{-1}}\D + \Exc(\phi) \sim_\Q \phi^*(K_X+\D) +\sum_{\parbox{33pt}{\tiny $E_i$  is exceptional}} a(E_i,X,\D)E_i +\Exc(\phi) 
\\ \geq \phi^*(K_X+\D) + \underbrace{\sum_{\parbox{75pt}{\tiny $E_i$  is exceptional and $a(E_i, X, \Delta) <-1$}} \Big(a(E_i,X,\D)+1\Big)E_i}_{=:-G},
\end{multline}
where $G \geq 0$ and  $X^0 \cap \Supp G = \emptyset$, as $(X^0, \Delta^0)$ is lc. Hence, we have the following inclusions and isomorphisms, where each inclusion is isomorphism over a dense open set of $T$:
\begin{equation}
\label{eq:stab_red_incl_1}
\mathcal E^{\D'}_{f', q} 
\explparshift{200pt}{-140pt}{\hookrightarrow}{$\Delta' \leq {\phi_{*}^{-1}}\D + \Exc(\phi)+G$ and the two sides of this inequality are equal over a non-empty open set of $T$}
\mathcal E^{{\phi_{*}^{-1}}\D + \Exc(\phi)+G}_{f', q} 
\explparshift{150pt}{130pt}{\cong}{the difference betweent he two sides of inequality \autoref{eq:stab_red_step_1} are all $f$-exceptional}
\mathcal E^\Delta_{f, q}
\end{equation}
By \autoref{lem:fibre_ridotte} we can take a finite map $\tau \colon S \to T$ such that the normalisation $Y'  \to X_S'$ has reduced fibers over $S$. Set $g': Y' \to S$ and $\xi: Y' \to X$ be the induced morphisms.

Since $(X', \D' + (X_t')_{\red})$ is lc for every $t \in T$, we can apply \cite[Lemma 2.52]{Kobook} to show that $(Y', \Theta' + Y_s')$ is lc for any $s \in S$,  where $\Theta'=\xi^*\D'$. 
The fibers $Y_s'$ being Cartier, we get that  $(Y',\Theta')$ is lc. Also, as $q \Delta'$ is a $\Z$-divisor, so is $q \Theta'$. Furthermore, observe that the generic fiber of $(Y',\Theta')$ is isomorphic to the generic fiber of $(X',\Delta')$.

Let $E$ be the vertical effective $\Z$-divisor on $Y'$ defined in Lemma \ref{lem:fibre_ridotte}, for which $\xi^* K_{X'/T} = K_{Y'/S} + E$. 
We apply \autoref{lem:reduction_bundle} with $L=q(K_{X'/T}+\Delta')$ (here we use that $q\Delta'$ is a $\Z$-divisor to apply \autoref{lem:reduction_bundle}):
\begin{align*}
L_{Y'}=& \xi^* (L-K_{X'/T}) + K_{Y'/T} 
= \xi^* \Big( (q-1) K_{X'/T} + q \Delta'\Big) + K_{Y'/T} \\
=& q\left(K_{Y'/S}+\Theta' + \frac{q-1}{q}E\right)
\end{align*}
Set $\D_{Y'}=\Theta + \frac{q-1}{q}E$, so that $L_{Y'}=q(K_{Y'/S}+\Delta_{Y'})$. Then, we have the following generically isomorphic inclusions of vector bundles on $S$:
\begin{equation}
\label{eq:stab_red_incl_2}
\mathcal E^{\Theta'}_{g', q} 
\explshift{-20pt}{\hookrightarrow}{$q\D_{Y'}- q\Theta'$ is a vertical effective integral divisor}
\mathcal E^{\Delta_{Y'}}_{g', q} 
\cong
g'_*\O_{Y'}(L_{Y'})
\explshift{30pt}{\hookrightarrow}{\autoref{lem:reduction_bundle}}
f'_* \O_{X'}(L) \cong \mathcal E^{\Delta'}_{f', q}
\end{equation}
As discussed before, $(Y', \Theta' + Y_s')$ is lc for any $s \in S$, so $g'$ has no vertical lc center and we can apply \cite[Theorem 11.28 (2)]{Kobook} to obtain the relative canonical model $g\colon(Y,\Theta)\to S$  of $g'\colon\left(Y', \Theta' \right) \to S$. Such model is stable by \cite[Proposition 2.47]{Kobook}. The generic fiber of the $g$ is birational to the generic fibers of $f$, the log canonical bundle is ample on both generic fiber, hence they are isomorphic. Because of the way $\Theta$ is constructed, the isomorphism preserve the boundary. We conclude that $(Y,\Theta)$ is isomorphic to $(X_S,\Delta_S)$ over $\tau^{-1}(T^0)$ by \cite[Proposition 2.50]{Kobook}. Additionally, by \cite[Corollary 1.25]{Kosing} we obtain that $\mathcal E^{\Theta'}_{g', q} \cong \mathcal E^{\Theta}_{g, q}$. Putting this together with the generically isomorphic inclusions of \autoref{eq:stab_red_incl_2} and \autoref{eq:stab_red_incl_1}, we obtain a generically isomorphic inclusion:
\begin{equation*}
    \mathcal E^{\Theta}_{g, q} \hookrightarrow  \mathcal E^{\Delta}_{f, q} 
\end{equation*}
By \autoref{lem:fuj2.2} this implies Item \ref{prop:stable_reduction_itemHodge}, and the inequality 
$$
\deg(\lambda_{g,q}^{\Theta})\leq \deg(\tau)\deg(\lambda_{f,q}^{\Delta}),
$$ which in turn implies Item \ref{prop:stable_reduction_itemLambda}. When for both families $f$ and $g$ the relative canonical bundle is relatively ample, we can apply \autoref{prop:lambda_and_CM} to deduce Item \ref{prop:stable_reduction_itemCM} from the inequality $\deg(\lambda_{g,q}^{\Theta})\leq \deg(\tau)\deg(\lambda_{f,q}^{\Delta})$.

\end{proof}
The following lemma is well-known and it has been used multiple times.
\begin{lemma}\label{lem:fuj2.2}
Let $T$ be a smooth projective curves, $W$ and $V$ two locally free sheaves on $T$, $\iota \colon W \rightarrow V$ a morphism which is surjective over an open dense subset of $T$. Then if $W$ is nef (resp. ample), then $V$ is nef (resp. ample). Moreover, $\deg\det(W)\leq \deg \det(V)$.
\end{lemma}
\begin{proof}
The first claim follows from the Barton-Kleiman Criterion \cite[Proposition 6.1.16]{Laz2}. For the second observe that $\det(\iota)\colon \det(W) \to \det(V)$ is a morphism of line bundle which is an isomorphism over an open dense subset of $T$.
\end{proof}

\subsection{Positivity}\label{S:positivity}

In the next proof we will use many times semi-snc resolutions. Following \cite[Definition 1.7, page 9]{Kosing} we say that a pair $(X,\Delta)$ is simple normal crossing (snc for short) if $X$ is smooth and the support of $\Delta$ is a snc divisor. Following  \cite[Definition 1.10 page 11]{Kosing}, we say that a pair $(X,\Delta)$ is semi-snc if for every closed point $x$ of $X$, there exist a snc pair $(W,\sum E_i+\sum F_i)$ such that a Zariski neighborhood of $x$ in $(X,\Delta)$ is isomorphic to $(Y,\sum a_i F_i|_Y)$, where $Y$ is the union of the divisors $E_i$'s. Note that in \cite{Fuj} our semi-snc pairs are called just snc pairs. A semi-snc pair is called \emph{double semi-snc}, if at most two irreducible components of X meet in any point, i.e. all points of $X$ have multiplicity at most two.  According to \cite[Cor 10.57]{Kosing}:

\begin{theorem}[Existence of double semi-snc resolutions]
\cite[Cor 10.57]{Kosing}
\label{thm:res_sing}
Let $X$ be a variety and $D$ a reduced Weil-divisor on $X$. Then there exists a semi-snc log-resolution $\rho : Y  \to X$. That is, $\Big(Y, \rho^{-1}_* D + \Ex(\rho)\Big)$ is double semi-snc.  Moreover we may choose $\rho$ such that:
    \begin{enumerate}
        \item  $\rho$ is an isomorphism over any open set $U \subseteq X$ such that $U$ has only regular and double snc points, $D|_U$ is smooth and $D|_U$ does not intersect the singular locus of $U$. 
        \item $\rho$ maps $Y_{\sing}$ birationally onto the closure of $U_{\sing}$.
    \end{enumerate}
\end{theorem}

Note the following lemmas too:

\begin{lemma}
In \autoref{thm:res_sing}, if $X$ is $S_2$ and the open set $U$  is a big open set (in each irreducible component of $X$), then the natural homomorphism $\O_X \to \rho_* \O_Y$ is an isomoprhism.
\end{lemma}

\begin{proof}
The natural homomorphism $\O_X \to \rho_* \O_Y$ is injective and isomorphism over $U$. Then the $S_2$ assumption together with the fact that $\rho_* \O_Y$ is $S_1$ (i.e. it has no embedded points) gives the global isomorphism. 
\end{proof}

\begin{lemma}
\label{lem:snc_Cartier}
Let $(X, D= \sum D_i)$ be a semi-snc pair  with reduced boundary, where the $D_i$ are irreducible. Let $\Delta= \sum_i a_i D_i$ for some $a_i \in \Q$. Then, the following are equivalent:
\begin{enumerate}
\item \label{itm:snc_Cartier:semi_snc} $(X, \Delta)$ is semi-snc, 
    \item \label{itm:snc_Cartier:Delta} $\Delta$ is $\Q$-Cartier,
    \item \label{itm:snc_Cartier:K_X_plus_Delta} $K_X + \Delta$ is $\Q$-Cartier, and
    \item \label{itm:snc_Cartier:coeffs} whenever $\Supp D_i \cap \Supp D_j$ contains a singular codimension $2$ point of $X$, we have $a_i = a_j$.
\end{enumerate}
\end{lemma}

\begin{proof}
As $X$ is semi-snc, $K_X$ is Cartier. Hence point \autoref{itm:snc_Cartier:Delta} and \autoref{itm:snc_Cartier:K_X_plus_Delta} are euqivalent. Points \autoref{itm:snc_Cartier:semi_snc} and \autoref{itm:snc_Cartier:K_X_plus_Delta} are equivalent by definition. So, it is enough to show that \autoref{itm:snc_Cartier:Delta} and \autoref{itm:snc_Cartier:coeffs} are equivalent.

First, note the following fact (from \cite[Item (2) after Definition 1.10 page 11]{Kosing}): if locally $X=\{xy=0\} \subseteq Z$ and $E= \{z=0\} \subseteq X$,  where $x, y, z$ are local parameters, then $(X,E)$ is semi-snc, and $E= E_1+ E_2$ has two components. In this case $\alpha_1 E_1 + \alpha_2 E_2$ is $\Q$-Cartier if and only if $\alpha_1 = \alpha_2$, which is shown in the rest of the proof.

Additionally by putting in further components of $X$ or for $E$, does not change the condition of $\alpha_1 E_1 + \alpha_2 E_2$ being $\Q$-Cartier. Indeed, if $\alpha_1 = \alpha_2$, then $z$ gives a global defining equation, independently of the presence of other components. On the other hand if $\alpha_1 \neq \alpha_2$, then at the generic point of $\{x=y=0\}$ there are no other components intersecting. Hence at these generic points, $\alpha_1 E_1 + \alpha_2 E_2$ is not $\Q$-Cartier.  
\end{proof}

\begin{lemma}
\label{lem:pullback_Q_divisor}
Let $\rho: Y \to X $ be a proper birational morphism between demi-normal varieties such that $\rho$ is isomorphism in codimension $1$ of $X$ and that no codimension 1 component of $Y_{\sing}$ is contracted by $\rho$. If $D$ is a $\Q$-Cartier divisor on $X$, then $\O_X(D) \cong \rho_* \O_Y(\rho^*D)$.
\end{lemma}

\begin{proof}
We may assume that $X$ is affine, and then it is enough to show that every element of $H^0(X,D)$ extends to an element of $H^0(Y, \rho^*D)$. Indeed, if this is the case, then we have an injection $\O_X(D) \hookrightarrow \rho_* \O_Y(\rho^*D)$ which is isomorphism in codimension $1$, because $\rho$ itself is an isomorphism in codimension $1$. Hence it is an isomorphism globally by the $S_2$ property.

To show the above extension property, just note that 
\begin{equation*}
H^0(X,D)=\{ s \in K(X)|(s)_X + D \geq 0 \} \subseteq \{ t\in K(Y)| (t)_Y + \rho^*D \geq 0 \} = H^0(Y, \rho^*D)
\end{equation*}
    by assigning $t=\rho^*s$.
\end{proof}

\begin{proposition}[Existence of double semi-snc principalizations]
\label{prop:principalization}
 Let $(X, \Delta)$ be a connected demi-normal pair which has simple normal crossing singularities in codimension $1$ and let $\mathcal{I}$ be an ideal sheaf on $X$ such that $\mathcal{I} = \O_X$ in a neighborhood of every singular codimension $1$ point. Then the double snc principalization of $\mathcal{I}$ exists.  That is, there exists a proper birational morphism $\rho: Y \to X$ such that
\begin{itemize}
    \item 
    $\rho$   does not contract any irreducible component of $Y$ or of $Y_{\sing}$,
    \item the sheaf $\rho^{-1} \mathcal{I}$ is locally principal, and defines a Cartier divisor $F$ on $Y$,
    \item the pair $(Y,F +  \Delta_Y)$ is semi-snc, where $\Delta_Y$ is the crepant boundary associated to $\rho$ and $\Delta$, and
    \item $\rho$ is an isomorphism over a big open set $U$ of $X$.
\end{itemize}

\end{proposition}

\begin{proof}
Before starting the proof note that  the connectedness assumption together with the condition that $\mathcal{I}= \O_X$ at the generic points of $X_{\sing}$ implies that $\mathcal{I}=\O_X$ generically in all irreducible components. 

 We construct $Y$ in multiple steps. First, we fix the big open set $U$ such that $\mathcal{I}|_U=\O_U(-F_U)$ for some Cartier Mumford divisor $F_U$ on $U$, and  such that $(U, \Delta|_U + F_U)$ is double semi-snc. To obtains such a  $U$ we start with the double semi-snc locus $V$ of $(X, \Delta)$; this open set is big because $X$ is a demi-normal pair which has simple normal crossing singularities in codimension $1$. Then we remove the intersection every non divisorial component of $\O_X/\mathcal{I}$ from $V$. Finally we remove the intersection of the divisorial components of $\O_X/\mathcal{I}$ and of $V_{\sing}$.
    
 Next, we define $\alpha : Z \to X$ to be the blow-up $\Bl_X \mathcal{I}$ of $\mathcal{I}$. Note that as $X$ is reduced, so is $Z$, and by the definition of $U$, $\alpha$ is an isomorphism over $U$. Furthermore, if $V \subseteq Z$ is an irreducible component, then $\alpha^{-1} Z\cong \Bl_Z \big( \mathcal{I}|_Z \big)$, and hence $\alpha^{-1} Z$ is integral.  Additionally $\alpha^{-1} \mathcal{I}$ is locally free. Here there is a little catch: $Z$ is only reduced and not neccessarily $S_2$ and $G_1$. So, we do not have the equivalence between Mumford divisorial  and almost Cartier divisorial sheaves.   Still, $\Supp \O_X/\alpha^{-1} \mathcal{I}$ is purely of codimension $1$ by Hauptidealsatz, and so we can think about its support as a reduced Weil divisor (i.e., linear combination of codimension 1 subvarieties of $Z$). Let us call $F_{Z, \red}$ this divisor. Note that, despite the failure of the bijection between almost Cartier divisorial sheaves and Weil divisors, by hauptidealsatz we have the property that $\Supp \big( \O_X/ \alpha^{-1} \mathcal{I} \big) = \Supp F_{Z, \red}$.   Also,  $F_{Z, \red}|_U=F_{U, \red}$, where , as $\alpha$ is an isomorphism, we think about $U$ being both an open set of $Z$ and of $X$. This is the only properties of $F$ that matter for the continuation of the proof.
    
    Replace then $Z$ and $\alpha$,
    with the birational model obtained by  precomposing $\alpha$ with the blow-up of $\Exc(\alpha)$. Note that this precomposition is identity over $U$, and hence $\alpha$ is still an identity over $U$. This way we obtain that both $\O_Z/\alpha^{-1}\mathcal{I}$ and $\Exc(\alpha)$ have pure codimension $1$, with reduced Weil divisors $F_{Z, \red}$ and $\Exc(\alpha)$ associated to them. 
    
    Let $\Gamma = \big(\alpha^{-1}_* \Delta+ F_{Z, \red} + \Exc(\alpha) \big)_{\red}$.  Let us apply then double semi-snc resolution to $(Z, \Gamma )$. 
 This yields a proper, birational morphism $\beta: Y \to Z$ such that $\big(Y, \Gamma':=\beta^{-1}_* \Gamma + \Exc(\beta) \big)$ is double semi-snc and such that $\beta$ is an isomorphism over $U$. We define $\rho= \alpha \circ \beta$. Note that as the support of $\Exc(\alpha)$ is contained in $\Gamma$, it follows that $\Exc(\rho)$ is contained in $\Gamma'$. However, as the irreducible components of the boundary of a double semi-snc pair are not contained in the singular locus, it follows that $\rho$ is an isomorphism at the generic point of every irreducible component $G$ of $Y_{\sing}$ (all such $G$ is of codimension $1$ by the semi-snc condition). In particular, all such $G$ intersects $\rho^{-1} U$ non-trivially, and hence $\rho^{-1} \mathcal{I}= \O_Y$ at the generic point of all such $G$.
 Hence, there is a Mumford divisor $F$ on $Y$ corresponding to the locally free sheaf $\rho^{-1} \mathcal{I}$ (it is locally free, because $\alpha^{-1} \mathcal{I}$ is already locally free).  That is, we have $\rho^{-1}\mathcal{I}=\O_Y(-F)$. As $\mathcal{I}$ is locally free,  $F$ is Cartier. Also, as $\beta^{-1}_* F_Z + \Exc(\beta)$ contains $F$, we see that $F_{\red} \leq \Gamma'$.

We now show that $(Y,F+\Delta_Y)$ is double-semi-snc. As $\Delta_Y$ is the crepant boundary to $\Delta$, we automatically have that $K_Y + \Delta_Y$ is $\Q$-Cartier. Also, as $\Supp \Delta_Y \subseteq \rho^{-1} \Delta + \Exc(\rho)$ we have that $\Supp \Delta_Y \subseteq \Supp \Gamma'$. Using now the above shown properties of $F$, we see that $K_Y + \Delta_Y + F$ is also $\Q$-Cartier and that $( \Delta_Y + F)_{\red} \subseteq \Gamma'$. Hence by \autoref{lem:snc_Cartier}, Item (3), we see that $(Y, \Delta_Y + F)$ is double semi-snc.

\end{proof}

The following result generalizes the construction of a Galois cover - see e.g. \cite[Section 2.44 page 67]{Kosing} - to demi-normal pairs; we will use it only in the case of double snc pairs.

\begin{lemma}[Construction of Galois covers of demi-normal pair]
\label{lem:cyclic_snc}
 Let $(X,G)$ be a demi-normal pair such that $G$ is effective and Cartier,  and such that $G \in |\mathcal{L}^t|$ for a line bundle $\mathcal{L}$  and an integer $t>0$, and the support of $G$ does not contain any codimension one component of the singular locus. Let $\Gamma$ be a Cartier divisor on $X$. Then there exists a finite $\Z/t\Z$-Galois cover $p : Z \to X$ by a demi-normal variety such that
 \begin{itemize}
     \item the codimension $1$ irreducible components of $Z_{\sing}$ are exactly the irreducible components of $p^{-1} X_{\sing}$,
     \item $p$ is étale over the codimension $1$ points that are not contained in $G$,
 \end{itemize} 
 and additionally
\begin{equation}
\label{eq:cyclic_snc:goal}
p_* \O_Z( K_Z + p^* \Gamma) \cong \bigoplus_{i=0}^t \O_X(K_X + \Gamma) \otimes \mathcal{L}^{(i)}
\end{equation}
where
\begin{equation*}
\mathcal{L}^{(i)}:=\mathcal{L}^{\otimes i}\otimes \O_X\left(-\left\lfloor\frac{iG}{t} \right\rfloor\right)
\end{equation*}
for every integer $i \geq 0$.
\end{lemma}

\begin{proof}

Consider the sheaf of $\O_X$-algebras
\begin{equation*}
\mathcal{A}=\bigoplus_{i=0}^{t-1} \mathcal{L}^{(-i)} 
\end{equation*}
where
\begin{equation*}
\mathcal{L}^{(-i)}:=\mathcal{L}^{\otimes -i}\otimes \O_X\left(\left\lfloor\frac{iG}{t} \right\rfloor\right)
\end{equation*}
for every integer $i \geq 0$, and the ring structure is explained in \cite[Claim 3.8]{EV92}. Additionally,  $\mathcal{A}$ is regular over all non-singular codimension $1$ points of $X$ \cite[Claim 3.12 on page 24]{EV92}. Note that $\mathcal{A}$ in this citation agrees with the above defined $\mathcal{A}$ \cite[pages 22--23]{EV92}. $\mathcal{A}$ is étale over the singular codimension one points because the support of $G$ does not contain any singular codimension one point. These, together with the $S_2$ property of $\mathcal{A}$, gives that $Z:= \Spec \mathcal{A} $ is demi-normal. Set $p :Z \to X $ be the structure morphism. 

By the projection formula, it is enough to show \autoref{eq:cyclic_snc:goal} for $\Gamma=0$. And in that case, it follows from the usual duality argument \cite[Propositions 5.67 (1) and 5.68]{KM}:
\begin{equation*}
p_* \omega_Z \cong p_* \sHom_Z(\O_Z, \omega_Z) \cong \sHom_X(p_* \O_Z , \omega_X) \cong \bigoplus_{i=0}^t \omega_X \otimes \mathcal{L}^{(i)} 
\end{equation*}

\end{proof}

We now recall a key semipositivity result from Hodge theory. These sort of statements have been known for a long time, see e.g. \cite[Corollary 3.7]{Kol86a} or \cite[Section 6.3.E]{Laz2} and references therein; for the generality we need here we refer to \cite[Theorem 1.9, first part of Theorem 1.10, and references therein]{Fuj}.

\begin{theorem}[Basic semipositivity]\label{basic_sem}
Let $f: (Z,D) \to T$ be a family where $(Z,D)$ is a projective double semi-snc pair, and $D$ is a reduced integral divisor whose support does not contain any irreducible component of the fibers of $f$. Then 
$
f_*\O_Z(K_{Z/T}+D)
$
is  nef.
\end{theorem}

We can now state and prove our effective semipositivity theorem for Hodge bundles, which generalizes \cite[Theorem 1.10]{Fuj}.

\begin{theorem}[Semipositivity of Hodge bundles]\label{thm: Fujino}
Let $f: (X,\D) \to T$ be a family where $(X,\D)$ is a projective demi-normal pair with $K_{X/T} +\D$ $\Q$-Cartier and $T$ a smooth projective curve.  Assume that the general fibre $(X_t,\D_t)$ of $f$ is slc. Then 
$$
f_*\O_X(q(K_{X/T}+\Delta))
$$
is  nef  for every positive integer $q$.  
\end{theorem}

\begin{proof}
We fix $q$ throughout the proof, as many of our constructions will depend on the value of $q$. We can assume that $f_*\O_X(q(K_{X/T}+\Delta))$ has strictly positive rank, otherwise the statement is trivial. Notation: $\Delta$ with subindex always denotes the boundary crepant with $\Delta$ on the corresponding birational model.

\begin{claim}\label{Cl:semi-snc}
We may assume that $(X,\Delta)$ is an slc\footnote{Note that slc in the semi-snc case means simply that the coefficients of $\Delta$ are not greater than $1$} double semi-snc pair, and $q \Delta$ is a $\Z$-divisor.
\end{claim}
\begin{proof}[Proof of Claim \ref{Cl:semi-snc}] First we reduce to the case when $(X,\Delta)$ is semi-snc in codimension one. To this end we take an irreducible component of the double cover from \cite[Subsection 5.23, page 203-204]{Kosing} $\pi \colon \tilde{X}\to X$. Observe that $f_*\O_X\Big(q(K_{X/T}+\Delta)\Big)$ is a direct summand of $(f \circ \pi)_*\O_{\tilde{X}}\Big(q(K_{\tilde{X}/T}+\pi^*\Delta)\Big)$ (more precisely, it is the invariant part for the involution associated to the Galois cover $\pi$), hence we can replace $(X,\Delta)$ with $\left(\tilde{X},\pi^*\Delta\right)$ and we may assume that $(X,\Delta)$ is (double) semi-snc in codimension one.

Take now a double semi-snc log-resolution  $\rho \colon Y \to X $ given \autoref{thm:res_sing}. By definition $\Delta_Y$ is the divisor that makes the equality $\rho^*(K_X+\Delta)=K_Y+\Delta_Y$ hold.  In particular, by \autoref{lem:snc_Cartier}, $(Y, \Delta_Y)$ is snc. 

Define $\Gamma'=\frac{1}{q}\left\lfloor q\Delta_Y^{>0}\right\rfloor $, and let $\Gamma$ be the divisor obtained from $\Gamma'$ by decreasing all the coefficients greater than $1$ to $1$. Note that if $\Delta_Y$ satisfies condition \autoref{itm:snc_Cartier:coeffs} of \autoref{lem:snc_Cartier}, then so does $\Gamma$ and $\Gamma'$. The reason is that both $\Gamma$ and $\Gamma'$ are defined solely by conditions on coefficients. So, in particular, equal coefficients stay equal during the operations performed in the definitions of $\Gamma$ and $\Gamma'$. 

Write $G = \Gamma' - \Gamma$, which is an effective divisor by definition. Then,  the pair $(Y, \Gamma )$ is slc and double semi-snc. Additionally if we set $g= f \circ \rho$, then we have  a generically isomorphic inclusion of vector bundles: 
\begin{multline*}
f_* \O_X\Big(q(K_{X/T}+\Delta)\Big)
\explparshift{300pt}{120pt}{=}{This equality works  even if $q \Delta$ and $q \Delta_Y$ are not $\Z$-divisors, by the equality $\rho^*(K_X+\Delta)=K_Y+\Delta_Y$ of $\Q$-divisors and by  \autoref{lem:pullback_Q_divisor}}
g_*\O_{Y}\Big(q(K_{Y/T}+\Delta_Y)\Big)
=g_*\O_{Y}\Big(q(K_{Y/T}+\Gamma')\Big)
\\ = g_*\O_{Y}\Big(q(K_{Y/T}+\Gamma+G)\Big)
\explparshift{300pt}{-150pt}{\hookleftarrow}{$G \geq 0 $ and it its support is contained in finitely many closed fibers as the general fiber of $(X,\Delta)$ over $T$ is slc}
g_*\O_{Y}\Big(q(K_{Y/T}+\Gamma)\Big)
\end{multline*}
This means that, by \autoref{lem:fuj2.2}, we may prove the theorem for $(Y, \Gamma)$ instead of $(X, \Delta)$. With other words we may assume that $(X, \Delta)$ is an slc double semi-snc pair and $q\D$ is a $\mathbb Z$-divisor.
\end{proof}

We continue the proof of \autoref{thm: Fujino}.
Note that $q(K_{X/T} + \Delta)$ is Cartier since $(X,\D)$ is double semi-snc and $q\D$ is a $\Z$-divisor. Hence, we may pull it back via birational morphisms. We use this multiple times in what follows. 

\begin{claim}
\label{claim:codim_one_cmp_sing_locus}
We may assume that the relative base-locus of $q(K_{X/T} + \Delta)$ does not contain any codimension one irreducible component of the singular locus of $X$.
\end{claim}

\begin{proof}
Assume it contains one codimension one irreducible component $E$ of $X_{\sing}$.  Consider the blow-up $p\colon (Y,\Delta_Y)\to (X,\Delta)$ of $E$. This blow-up is the is a partial normalization at $E$ separating the two irreducible components meeting at $E$; to show it, observe that since $(X,\Delta)$ is double semi-snc, all points are either regular or locally of the for $xy=0$, and hence the blow-up can be explicitly described.  In particular, $p$ is finite, and $(p^{-1} E)_{\red}$ is the upstairs conductor $D$ of $p$, which has two irreducible connected components. The downstairs conductor of $p$ is $E$ itself. We also see that $Y$ is regular along $D$, and hence $D$ is Cartier. In particular, the crepant boundary $\Delta_Y$ is the strict transform of $\Delta$ plus $D$. Set $\Gamma = \Delta_Y - D=p^{-1}_* \Delta$. This is $\Q$-Cartier, as $D$ is Cartier, and hence $(X, \Gamma)$ is (double) semi-snc by \autoref{lem:snc_Cartier}.  Then, we have
\begin{multline*}
  (f\circ p)_*\O_{Y}\Big(q(K_{Y/T}+\Gamma)\Big)   
  =
(f \circ p )_*\bigg( \O_Y (-D) \otimes  p^* \O_X  \Big( q(K_{X/T}+\Delta)\Big) \bigg) 
\\  \explshift{-60pt}{\cong}{projection formula}
f_*\bigg(     \O_X  \Big( q(K_{X/T}+\Delta)\Big) \otimes p_*\O_Y (-D)\bigg) 
\explparshift{200pt}{-60pt}{=}{the conductor ideals $\O_Y(-D)$ and $\O_X(-E)$, upstairs and downstairs respectively, are identified by $p_*( \_ )$ }
\\
=f_*\bigg(     \O_X  \Big( q(K_{X/T}+\Delta)\Big) \otimes \O_X(-E) \bigg) 
\explparshift{90pt}{80pt}{=}{$E$ is in the relative base locus of $q(K_{X/T} + \Delta)$}
f_*   \O_X  \Big( q(K_{X/T}+\Delta)\Big) 
\end{multline*}

Thus, by replacing $(X, \Delta)$ with $(Y, \Gamma)$, we may  assume that the relative base locus of $q(K_{X/T} + \Delta)$ does not contain any codimension one component of the singular locus of $X$. Note that the assumptions that $(X,\Delta)$ is slc semi-snc and that $q \Delta$ is a $\Z$-divisor stay valid by this replacement. 
\end{proof}

\begin{claim}
We may additionally assume that the relative base-locus of $q(K_{X/T} + \Delta)$ is divisorial, and if $E$ is its reduced divisor, then $(X, E + \Delta)$ is double semi-snc and additionally $E$ is Cartier. 
\end{claim}

\begin{proof}
Let $\mathcal{I}$ be the relative base locus of  $q(K_{X/T} + \Delta)$, that is, we have
\begin{equation*}
\im \bigg( f^*f_*\O_X\Big(q(K_{X/T}+\Delta)\Big) \otimes\O_X\Big(-q(K_{X/T}+\Delta)\Big)\to \O_X \bigg) = \mathcal{I}
\end{equation*}
We take  $\rho \colon Y \to X$ be a double semi-snc principalization of $\mathcal{I}$ given by \autoref{prop:principalization}. Note that we can apply \autoref{prop:principalization} because of \autoref{claim:codim_one_cmp_sing_locus}. Let $E$ be the Mumford divisor on $Y$ corresponding to $\rho^{-1}\mathcal{I}$.
Then, $M_Y = \rho^* q(K_{X/T} + \Delta) - E$ is a relatively free linear system, and we have 
\begin{equation*}
q \big(K_{Y/T} + \Delta_Y^{>0}\big) = M_Y + E  -q \Delta_Y^{<0},
\end{equation*}
where $E  - q\Delta_Y^{<0}$ is the relatively fixed part of $q (K_{Y/T} + \Delta_Y^{>0})$,
and 
\begin{equation*}
(f\circ \rho)_* \mathcal{O}_Y\Big(q \big(K_{Y/T} + \Delta_Y^{>0}\big)\Big) = f_* \mathcal{O}_X\big( q (K_{X/T} + \Delta)\big). 
\end{equation*}
Note additionally, that $K_Y + \Delta_Y$ is $\Q$-Cartier, and hence by \autoref{lem:snc_Cartier} $(Y, \Delta_Y)$ is double semi-snc. As $\Delta_Y^{>0}$ is obtained from $\Delta_Y$ by a rule of coefficients (i.e., irreducible components in the boundary with the same coefficients in $\Delta_Y$ have same coefficients in $\Delta_Y^{>0}$ too), \autoref{lem:snc_Cartier} tells us that $\big(Y, \Delta_Y^{>0} \big)$ is also double semi-snc. Hence, we can conclude our claim  by replacing $(X, \Delta)$ with $\big(Y, \Delta_Y^{>0}\big) $. Note that then $E$ is going to be $E  - \Delta_Y^{<0}$. 
\end{proof}

From now on, we denote the relative fixed part of $q(K_{X/T} + \Delta)$ by $E$. 

\begin{claim} 
\label{claim:coeff_1_E}
We may assume that $E$ avoids all the coefficient $1$ components of $\Delta$.

\end{claim}
\begin{proof}
    We define $\Delta'$ as follows:
we  take those components $\Delta$ that have coefficient $1$ and that are contained in the support of $E$ and we may decrease their coefficients in $\Delta$ to $\frac{q-1}{q}$ from $1$, and their coefficient in $E$ by $1$.  
As the  components with decreased coefficients were in $E$, we have
\begin{equation*}
    f_* \mathcal{O}_X(q (K_{X/T} + \Delta)) = f_* \mathcal{O}_X(q (K_{X/T} + \Delta'))
\end{equation*}
We are only left to show that $(X, \Delta')$ is also double semi-snc. For that take two components $D_i$ and $D_j$ of $\Delta+E$ such that $D_i \cap D_j$ contains a codimension two singular point of $X$. Let $a_i$ and $a_j$ be their respective coefficients in $\Delta$ and let $b_i $ and $b_j$ be their respective coefficients in $\Delta'$. According to \autoref{lem:snc_Cartier} we have to show that if $a_i=a_j$, then $b_i=b_j$. So, really the only case, we should be worried about is $a_i =a_j=1$. In this case, we should show that either both or none of $D_i$ and $D_j$ are contained in $\Supp E$. To prove it, observe that separately $D_i$ and $D_j$ are not Cartier, since $E$ is Cartier again by \autoref{lem:snc_Cartier} either it contains both or them of none of them.

\end{proof}

We denote the relative mobile part of $q(K_{X/T}+\Delta)$ by $M$; it is relatively free by the above assumptions.

\begin{lemma}\label{lem:bas_red} 
    
In the above setting, let $\A$ be a Cartier divisor on $T$ such that $\O_X(M)\otimes f^*\A^{\otimes q}$ is semiample on $X$, then $f_*\O_X(q(K_{X/T}+\Delta))\otimes \A^{\otimes (q-1)}$ is a nef vector bundle on $T$.

\end{lemma}

\begin{proof}
In the following proof we will multiple times modify the boundary of double semi-snc pairs by coefficient rules. That is, if coefficients were equal before the modification, they stay equal after the modification. By \autoref{lem:snc_Cartier}, such modifications turn (double) semi-snc pairs into (double) semi-snc. All modifications of the boundary in the following proof is of this type, which fact we will not mention necessarily again at every change of boundary separately. 
    
Take a positive integer $N$ such that $(\O_X(M)\otimes f^*\A^{q})^{\otimes N}$ is base point free, and take a general member $H$. By Bertini Theorem, $(X,E + \Delta + H)$ is semi snc.

We consider the divisorial sheaf
$$
\mathcal{L}=\O_X(K_{X/T}+\Delta^{=1}+q\{\Delta\})\otimes f^* \mathcal{A}
$$ 
As $\Delta^{=1} + q \{\Delta \}$ has integer coefficients and is obtained from $\Delta$ with a coefficient rule, we see that it is in fact Cartier by  \autoref{lem:snc_Cartier} . Hence $\mathcal{L}$ is a line bundle. 

Observe that $G:=H+NE+Nq(q-1)\{\Delta\}$ is a  section of $\mathcal{L}^{\otimes qN}$ contained in the support of the boundary of $(X,E + \Delta + H)$. Note then the following properties of $G$ and of $\Delta^{=1}$:
\begin{itemize}
    \item $G$ can be obtained from $\Delta$, $E$ and $H$ using coefficient rules. As $X$ becomes semi-snc with all three of these as boundaries, we see that $(X, G)$ is also semi-snc. For the seame reasons $(X, G + \Delta^{=1})$ is also semi-snc  
    \item $G$ avoids the codimension one singular points of $X$ by \autoref{claim:codim_one_cmp_sing_locus}.
    \item  $\Delta^{=1}$ is Cartier (it has coefficients $1$ and it is obtained by coefficient rules from $\Delta$). 
    \item $G$ and $\Delta^{=1}$ have no common components by the genericity of $H$ and by \autoref{claim:coeff_1_E}.
\end{itemize}
Hence, we may apply
\autoref{lem:cyclic_snc} to get  a finite cyclic $\Z/qN \Z$-Galois cover $p : X \to T$ from a demi-normal variety such that $p$ is \'etale over the generic points of $\Delta^{=1}$ and such that 
\begin{equation}
\label{eq:eigensheaves}
p_* \O_X(K_Z + p^* \Delta^{=1}) \cong \bigoplus_{i=0}^{qN-1} \O_X(K_X + \Delta^{=1})  \otimes \mathcal{L}^{(i)},
\end{equation}
where
$$
\mathcal{L}^{(i)}:=\mathcal{L}^{\otimes (i)}\otimes \O_X\left( -\left\lfloor\frac{i}{Nq}G \right\rfloor\right)
$$
By the above unramifiedness over the generic points of $\Delta^{=1}$, we have a crepant relation  
\begin{equation}
\label{eq:cyclic_cover_crepant}
    p^* (K_X + \Delta^{=1} + \Gamma)= K_Z + p^* \Delta^{=1} 
\end{equation} for an adequate effective weil divisor $\Gamma$ with support in $\Supp G$, without common components with $\Delta^{=1}$. More precisely, $\Gamma$ is the ramification divisor, and hence all its components have coefficients less than $1$. Note that $\Gamma$ can be defined by a coefficient rule from $G$, and hence $(X, \Gamma + \Delta^{=1})$ is semi-snc. As all the coefficients of $\Gamma + \Delta^{=1}$ are at most $1$, $(X, \Gamma + \Delta^{=1})$ is slc, and then  so is $(Z, p^* \Delta^{=1})$  by \autoref{eq:cyclic_cover_crepant}. 

Therefore we may apply \autoref{Cl:semi-snc} and \autoref{basic_sem} to $(Z, p^* \Delta^{=1})$, and obtain that $(f\circ p)_*\O_{Z}(K_{Z/T}+p^*(\Delta^{=1}))$ is nef.

By \autoref{eq:eigensheaves}, $f_*\left(\O_X(K_{X/T}+\Delta^{=1})\otimes \mathcal{L}^{(q-1)}\right)$ is also nef. As
$$
\mathcal{L}^{(q-1)}=\mathcal{L}^{\otimes (q-1)}\otimes \O_X\left(-\left\lfloor \frac{q-1}{q}(E+q(q-1)\{\Delta\})\right\rfloor\right)
$$
we have 
\begin{align*}
& \O_X(K_{X/T}+\Delta^{=1})\otimes \mathcal{L}^{(q-1)}=\\
&= \O_X\left(q(K_{X/T}+\Delta)- \left\lfloor \frac{q-1}{q}E \right\rfloor  +q(q-2)\{\Delta\}-\lfloor(q-1)^2\{\Delta \}\rfloor \right) \otimes f^* \A^{\otimes (q-1)}
\end{align*}
Observe that $q(q-2)\{\Delta\}-\lfloor(q-1)^2\{\Delta \}\rfloor=0$ because $q\Delta$ is integral, hence we have a natural inclusion 
\begin{equation*}
    f_*(\O_X(K_{X/T}+\Delta^{=1})\otimes \mathcal{L}^{(q-1)}) \hookrightarrow f_*\O_X(q(K_{X/T}+\Delta))\otimes \A^{q-1}
\end{equation*}
Since $\lfloor\frac{q-1}{q}E\rfloor \leq E$, and $E$ is the relative base locus of $\O_X(q(K_{X/T}+\Delta))$, then the above is an isomorphism and hence $f_*\O_X(q(K_{X/T}+\Delta))\otimes \A^{q-1}$ is nef. This conclude the proof of Lemma \ref{lem:bas_red}.
\end{proof}

\begin{lemma}\label{lem:bootsrap}
Fix an ample line bundle $\mathcal{H}$ on $T$. Then $f_*\O_X(q(K_{X/T}+\Delta))\otimes \mathcal{H}^{q^2-1}$ is nef. 

\end{lemma}
\begin{proof}
Let $r$ be the smallest positive integer such that $f_*\O_X(q(K_{X/T}+\Delta))\otimes \mathcal{H}^{rq-1}$ is nef. We thus have that for some $N$ big enough $\Sym^N f_*\O_X(q(K_{X/T}+\Delta))\otimes \mathcal{H}^{rq}$ is generated by global sections. We can now apply Lemma \ref{lem:bas_red} with $\mathcal{A}=\mathcal{H}^r$, and obtain that $f_*\O_X(q(K_{X/T}+\Delta))\otimes \mathcal{H}^{r(q-1)}$ is nef. If $r>q$, then $r(q-1)\leq (r-1)q-1$, which contradicts the minimality of $r$,  hence $r\leq q$ and $f_*\O_X(q(K_{X/T}+\Delta))\otimes \mathcal{H}^{q^2-1}$ is nef. 
\end{proof}

Let $\Sigma\subset T$ the locus over which the fibers of $f$ are not slc. For every surjective map $\tau \colon T' \to T$ which does not ramify at $\Sigma$, because \autoref{prop:flat_base_change} we have

\begin{equation}\label{e:b_tau}
\tau^*f_*\O_X(q(K_{X/T}+\Delta))=h_*\O_{X_{T'}}(q(K_{X_{T'}/T'}+\Delta_{T'}))\,.
\end{equation}

 We now run the full proof for $h\colon (X_{T'},\Delta_{T'})\to T'$ instead than for $f\colon (X,\Delta)\to T$, and in particular \autoref{lem:bootsrap}, combined with the base change \autoref{e:b_tau}, shows that for every ample line bundle $\mathcal{H}$ on $T'$ we have that $\tau^*f_*\O_X(q(K_{X/T}+\Delta))\otimes \mathcal{H}^{\otimes (q^2-1)}$ is nef. Now \autoref{2.3Fuj} with $\mu=q^2-1$ implies that $f_*\O_X(q(K_{X/T}+\Delta))$ is nef.

\end{proof}
The following Lemma, more or less well-known to experts, is used in the above proof, and it is proven in \cite[Lemma 2.3]{Fuj}.
\begin{lemma}\label{2.3Fuj} 
Let $T$ be a smooth projective curve, $V$ a vector bundle on $T$, and $\Sigma$ a proper closed subset of $T$. Assume that there exists a positive integer number $\mu$ such that for  every finite map $\tau\colon T' \to T$ from a smooth curve $T'$ which does not ramify at $\Sigma$, $\tau^*V \otimes \mathcal{H}^{\otimes \mu}$ is nef for every ample line bundle $\mathcal{H}$ on $T'$. Then $V$ itself is nef.
\end{lemma}

\begin{corollary}[Semipositivity of relative log canonical bundle, lambda classes and CM line bundle]\label{cor:semi_posit_lambda_e_CM}
Let $B$ be a normal projective variety, $(X,\Delta)$ a projective pair, $f\colon (X,\Delta) \to B$ be a faithfully flat map, $K_{X/B}+\Delta$ is $\Q$-Cartier, and assume we are in one of the three set-ups
\begin{enumerate}
    \item\label{is} $f\colon (X,\Delta) \to B$ is a  stable family, and $q \in \mathbb Z_{>0}$ is such that  $q\Delta$ is integral;
    \item\label{ic} $X$ is normal, the generic fiber of $f$ is slc, $K_{X/B}+\Delta$ is $f$-ample, $B$ is a smooth curve and $q$ is a strictly positive integer;
    \item\label{ibc} for every normal projective curve $T$ and every morphism $g\colon T \to B$,  the base change 
        $$g^*f_* \O_X(q(K_{X/B}+ \Delta))\cong h_*\O_{X_T}(q(K_{X_T/T}+ \Delta_T))$$ 
        holds, where $h\colon (X_T,\Delta_T)\to T$ is the base change of $f$.
\end{enumerate}
Then $\E_{q,f}^{\Delta}$, $K_{X/B}+\Delta$, $\lambda_{q,f}^{\Delta}$ and $\lambda_{CM,f}^{\Delta}$ are nef. 

Moreover, with the notations of \autoref{moduli}, $\lambda_{CM}$ is nef on $\M$ and, if $qI\subset \Z$, $\lambda_q$ is nef on $\M^{\nu}$.
\end{corollary}
\begin{proof}
When $B$ is a curve as in item \ref{ic}, $\E_{q,f}^{\Delta}=f_*\O_X(q(K_{X/B}+\Delta))$ is nef by \autoref{thm: Fujino}. As nefness can be checked on curves, when we have the base change assumed in item \ref{ibc}, the nefness follows again from \autoref{thm: Fujino}. In the case of a stable family from item \ref{is}, we have the base change as in item \ref{ibc} thanks to \autoref{prop:base_change_stable_families}, so we can apply again \autoref{thm: Fujino}.  

The above arguments also implies the nefness of $\lambda_{q,f}^{\Delta}$, as it is the determinant of $\E^{\Delta}_{q,f}$ . For $\lambda_{CM,f}^{\Delta}$, we can apply \autoref{prop:lambda_and_CM}, see also \cite[Corollary 2.14]{PX}.

As $K_{X/B}+\Delta$ is $f$-semiample, $\O_X(q(K_{X/B}+\Delta))$ is a quotient of  
$$f^*f_*\O_X(q(K_{X/B}+\Delta))$$ 
for $q$ big and divisible enough, so the nefness of $K_{X/B}+\Delta$ follows from the nefness of $f_*\O_X(q(K_{X/B}+\Delta))$.

To prove nefness on $\M$ and $\M^{\nu}$ we have to show that for every smooth projective curve $T$ and every map $g\colon T \to \M$ or to $\M^{\nu}$, the pull-back of the line bundle is nef. This is true by base change and item \autoref{is} of this corollary.

\end{proof}

\begin{remark}[Comparison with the Fano case] For families of Fano varieties, a statement analogue to \autoref{cor:semi_posit_lambda_e_CM} is false. The negativity of $-K_{X/T}$ is shown in \cite[Appendix]{Maciek}; the negativity of $f_*\O_X(-qK_{X/T})$ is studied more or less implicitly in many papers such as \cite{CP,Maciek,Xu_Zhuang}.
    
\end{remark}

Observe that, contrary to \autoref{cor:semi_posit_lambda_e_CM}, the following result on $\lambda_q$ is not effective in $q$.

\begin{theorem}[Positivity of the relative log canonical bundle, of the CM line bundle and of lambda classes]\label{thm:positivity_lambda_e_CM}
Under the same hypotheses of \autoref{cor:semi_posit_lambda_e_CM}, we have that
\begin{enumerate}

    \item\label{itemCM_pos} $f$ has maximal variation if and only if $\lambda_{CM}$ is big; 
    \item\label{itemq_pos} $f$ has finite isomorphism classes if and only if $\lambda_{CM}$ is ample, if and only if for all $q$ big and divisible enough $\lambda_q$ is ample;    \item\label{itemK_pos} if the family is stable, and  $K_{X/B}+\Delta$ is big, then the family has maximal variation.
\end{enumerate}

\end{theorem}
\begin{proof}

For items \autoref{itemCM_pos} and \autoref{itemq_pos}, when the base is a curve, we can assume that the family is stable using the stable reduction from \autoref{prop:stable_reduction}. From now on, we assume that the family stable.

Consider the moduli map $m_f\colon B \to M$ to the coarse moduli space of stable pairs. This map is finite/birational if and only if the family has finite isomorphism classes/maximal variation. 

The Chow-Mumford line bundle and $\lambda_q$ for $q$ divisible enough are $\Q$-Cartier divisors on $M$; on $B$, for $q$ divisible enough, they are pull-backs via $m_f$. 

The main result of \cite{PX} (resp. \cite{KP}) guarantees that $\lambda_{CM}$ (resp. $\lambda_q$ for $q$ big and divisible enough) are ample on $M$, so we have the claims \autoref{itemCM_pos} and \autoref{itemq_pos}.

For \autoref{itemK_pos}, if $K_{X/B}+\Delta$ is not big, then the family does not have maximal variation because of \cite[Corollary 6.20]{KP}, see also \cite[Lemma 2.19]{HJLL}.

\end{proof}

For the sake of completeness, let us recall that for stable family with maximal variation and lc generic fiber, $K_{X/B}+\Delta$ is big by \cite[Proposition 2.15]{PX}.

\begin{remark}[Negativity]
If we assume that $B$ contains a curve $T$ over which all fibers are not slc, it might well be that $\lambda_q$ and $\lambda_{CM}$ have strictly negative degree along $T$. 
\end{remark}

\section{Semipositive engine}

We will need the following version of Nadel vanishing theorem for  $\Q$-Cartier $\Z$-divisors; it is a variant of \cite[Lemma 3.1]{ST}.

\begin{lemma}\label{lem:Nadel}
	Let $(X, B)$ be a normal pair and $D$ a $\mathbb Q$-Cartier $\Z$-Weil divisor on $X$ such that $A=D-K_X-B$ is nef and big. Let $\mathcal J=\mathcal J((X,B); -D)$ be the multiplier ideal sheaf associated to $-D$ with respect to $(X,B)$.  Then
	\begin{enumerate}
		\item \label{lem:Nadel1} there is an inclusion $\mathcal J \hookrightarrow \O_X(D)$;
		\item \label{lem:Nadel2} $H^i(X, \mathcal J)=0$ for any $i >0$;
		\item \label{lem:Nadel3} Let $x\in X$ such that $\O_X(D)_{x} \cong \O_{X,x}$, i.e.\ $D$ is Cartier at $x$. Then $\mathcal J((X,B); -D)_x = \mathcal J(X,B)_x \otimes \O_X(D)_x$, where $\mathcal J(X,B):= \mathcal J(X, B; 0)$; 
	    \item \label{lem:Nadel4} Let $x\in X$ be a klt point of $(X,B)$. Then $\mathcal J((X,B); -D)_x = \O_X(D)_x$.
\end{enumerate}	 	
\end{lemma}

\begin{proof}
	Let $\mu \colon W \rightarrow X$ be a log resolution of $(X,B + D)$  and define a $\Q$-divisor $B_W$ by 
	$$
	K_W + B_W = \mu^*(K_{X} + B). 
	$$
	We can write $\mu^*D = \tilde{D}+\sum_{k=1}^m b_k E_k$, where $\tilde{D} \subset W$ is the strict transform of $D$ and 
	$E_1, \ldots, E_m$ are exceptional divisors of $\mu$. Since $D$ is a $\Z$-divisor, $\tilde{D}$ is a Cartier divisor on $W$.
	Set 
	$$
	\mathcal L:= \O_W\left(\tilde{D} + \left\lceil \sum b_k E_k - B_W \right\rceil\right)
    =
    \O_W\left(\tilde{D} - \left\lfloor -\sum b_k E_k + B_W \right\rfloor\right). 
	$$
	Then by definition of multiplier ideal sheaf (see \cite[Definition 9.3.56]{Laz2})
	$$
	\mathcal J=\mathcal J((X,B);-D)= \mu_* \mathcal L.
	$$
	If $E$ is the exceptional locus of $\mu$ and $Z:=\mu(E)$, we have  
	$$
	\mathcal L_{|(W\setminus E)} \cong \O_{X \setminus Z}(D + \lceil -B \rceil) \hookrightarrow \O_{X \setminus Z}(D).
	$$
	Hence we have $\mu_*\mathcal L \hookrightarrow (\mu_* \mathcal L)^{\vee \vee} \simeq \O_{X}(D + \lceil -B \rceil) \hookrightarrow \O_X(D)$ and obtain an injection $\mathcal J \hookrightarrow \O_X(D)$ as a composition. This shows \autoref{lem:Nadel1}.
	
	We now prove point \autoref{lem:Nadel2}. Since 
	$$
	\tilde{D} + \sum b_k E_k - B_W  \equiv \mu^*(D) -B_W \equiv \mu^*(K_{X} + B + A) - B_W = K_W+ \mu^*(A), 
	$$ 
	the relative Kawamata-Viehweg vanishing \cite[Theorem 1-2-3]{KMM87} implies
	$$
	R^i\mu_*(\mathcal L)=0
	$$
	for $i>0$ and so by Leray spectral sequence we get that $H^i(X, \mathcal J)= H^i(W, \mathcal L)$. By Kawamata-Viehweg vanishing we also get 
	$H^i(W, \mathcal L)=0$ for $i >0$ and so \autoref{lem:Nadel2} is proven. 
	
	Assume now that $D$ is Cartier at $x \in X$. Then 
	\begin{align*}
	\mathcal J_x & =\mu_*\mathcal O_W(\tilde{D} + \lceil \sum b_k E_k - B_W \rceil)_x=\mu_*\mathcal O_W(\mu^*D + \lceil - B_W \rceil)_x \\
    & = \mathcal J(X,B)_x \otimes \O_X(D)_x
	\end{align*}
	by the projection formula and we get \autoref{lem:Nadel3}.
	
	To prove \autoref{lem:Nadel4}, consider a neighbourhood $U$ of $x \in X$ such that $(U, B_U)$ is klt and $\mu: W_U \to U$ the induced resolution. Write 
	$$
	K_{W_U} + B_{W_U} = \mu^*(K_{U} + B_U). 
	$$
    Since $(U, B_U)$ is klt $\lceil - B_{W_U} \rceil$ is an effective exceptional integral divisor  and so 
    $$
    \mathcal J_x=\mu_*\mathcal O_{W_U}(\tilde{D} + \lceil \sum b_k E_k - B_U \rceil)_x=\mu_*\mathcal O_{W_U}(\lceil \mu^*D - B_U \rceil )_x=  \O_U(D)_x.
    $$
   The next to last equality is because $\tilde{D}$ is integral. 

The last inequality follows from Lemma \ref{lem:debarre} noting that
$$
\lceil \mu^*D \rceil \le \lceil \mu^*D -B_U \rceil \le \lceil \mu^*D \rceil + \lceil -B_U \rceil.
$$
\end{proof}

The following standard lemma is used in the above proof.

\begin{lemma}\label{lem:debarre}
Let $\pi\colon V\to W$ be a birational morphism of demi-normal projective varieties, let $E\subset V$ be an effective $\Q$-divisor such that $\pi(\Supp(E))$ has codimension two in $W$. Let $D$ be an integral $\Q$-Cartier divisor on $W$, then the pull-back $\pi^*$ gives isomorphisms
$$H^0(W,D) \cong H^0(V,\lfloor \pi^*D \rfloor+E) \cong H^0(V,\lceil \pi^*D \rceil+E)  $$
\end{lemma}
\begin{proof}		
Since $D$ is integral, $F=\lceil \mu^*D \rceil - \lfloor \mu^*D \rfloor$ is an effective exceptional divisor.
	
We have inclusions
\begin{align*}
H^0(W,D) &= H^0(V,\lfloor \pi^*D \rfloor) \subseteq H^0(V,\lfloor \pi^*D \rfloor +E) \subseteq H^0(V,\lceil \pi^*D \rceil +E) \\ & \subseteq H^0(V\setminus \Supp(E), \lceil \pi^*D \rceil) \cong H^0(W\setminus \pi(\Supp(E)),\lceil D \rceil) \\
& = H^0(W\setminus \pi(\Supp(E)), D) 
\end{align*}
	
 Since $W$ is demi-normal, the sheaf $\O_W(D)$ is S2, and we have
$$
H^0(W\setminus \pi(\Supp(E)),D)= H^0(W,D)  
$$
so the claim.
\end{proof}

The following semipositivity is a variant of \cite[Proposition 6.4]{CP}: we generalize loc. cit. to $\Q$-Cartier $\Z$-divisors instead of just Cartier divisors, but we need to assume that all fibers are reduced.

\begin{proposition}\label{p_nef_Weil}
Let $f \colon (Z, \Delta) \to T$  be a fibration from a normal projective pair to a smooth curve with $K_Z + \Delta$ $\Q$-Cartier. Let $L$ be a $\mathbb{Q}$-Cartier $\Z$-divisor on $X$ such that $L- K_{Z/T} - \Delta$ is $f$-big and globally nef on $X$.  Assume that the general fibre of $f$ is klt and all fibres are reduced.
  
Then, $f_* \mathcal{O}_Z(L)$ is a nef vector bundle. 
\end{proposition}

\begin{proof}
The fact that all fibres are reduced guarantees that $Z^{(m)}$ is normal for every $m$, see \autoref{lem:normal}, and that we can apply K\"{u}nneth formula from \autoref{lem:kunneth}. According to \cite[Lemma 3.4]{Pat14}, it is enough to prove that for all integers $m>0$, the following vector bundle is generated at a general $t \in T$ by global sections: 
$$
V_m:=\omega_T (2t) \otimes \bigotimes_{i=1}^m f_* \O_Z(L) \,.
$$
We fix a $t$ such that the fiber $Z_t$ is not contained in the support of $\Delta$, and the pair $(Z_t,\Delta_t)$ is klt. Let 
$$N:=L^{(m)} + \left( f^{(m)} \right)^* \left(K_T + 2 t\right)\,.$$
The divisor $N$ is not Cartier, but $K_T+2t$ is Cartier and thus we can apply projection formula for $f^{(m)}$ together with \autoref{lem:kunneth} to get an isomorphism $V_m\cong f^{(m)}_*\O_{Z^{(m)}}(N)$. To prove that $V_m$ is generically globally generated we need  to prove that the restriction map
$$
H^0\left( Z^{(m)}, N \right) \to H^0\left( Z_t^{(m)}, N_t \right)
$$
is surjective. 

Let $\mathcal J= \mathcal J\Big(\big(Z^{(m)},\Delta^{(m)}\big);-N \Big) $ be the multiplier ideal sheaf associated to $-N$ with respect to $\big(Z^{(m)},\Delta^{(m)}\big)$. 

Let us show that the following sequence is exact.
\begin{equation}\label{eq:short_exact_CM}
0 \to \mathcal J \otimes \O_{Z^{(m)}}\left(-Z^{(m)}_t\right) \to \mathcal{J} \to \mathcal{J}|_{Z^{(m)}_t}= \O_{Z^{(m)}_t}(N_t) \to 0.
\end{equation}
\

The fibre $Z^{(m)}_t$ is klt because it is the product of klt varieties, and so $Z^{(m)}$ is klt around it by inversion of adjunction. By Lemma \ref{lem:Nadel}\autoref{lem:Nadel4} this implies that 
$$
\mathcal J\Big(\big(Z^{(m)},\Delta^{(m)}\big);-N \Big)= \O_{Z^{(m)}}(N)
$$
in a neighbourhood of $Z^{(m)}_t$ and so in such neighbourhood the sequence reduces to 

$$
0 \to \O_{Z^{(m)}}\left(N-Z^{(m)}_t\right) \to \O_{Z^{(m)}}(N) \to \O_{Z^{(m)}_t}(N_t) \to 0.
$$
The latter is exact, because on a klt variety all divisorial sheaves are CM (\cite[Corollary 5.25]{KM}) and hence the sequence makes sense (see for instance \cite[Proposition 5.26]{KM}).  
Away from $Z^{(m)}_t$ the sequence \autoref{eq:short_exact_CM} is trivially true, so it is exact on all $Z^{(m)}$.

Since $H^0\left( Z^{(m)}, J \right) \hookrightarrow H^0\left( Z^{(m)}, N \right)$ by \autoref{lem:Nadel}\autoref{lem:Nadel1}, it is enough to prove the vanishing of 
\begin{equation*}
    H^1\left(Z^{(m)},\mathcal J \otimes \O_{Z^{(m)}}\left(-Z^{(m)}_t\right)\right)
    =
    H^1\left(Z^{(m)},\mathcal J\left(Z^{(m)},\Delta^{(m)}\right);-N+Z^{(m)}_t \right)
\end{equation*} 
(the equality follows by \autoref{lem:Nadel}\autoref{lem:Nadel3}, as $Z^{(m)}_t$ is Cartier). Write 
 $$N-Z_t^{(m)} = K_{Z^{(m)}} +  \Delta^{(m)} +  \left(L^{(m)} - K_{Z^{(m)}/T} - \Delta^{(m)} +  Z_t^{(m)}\right)\,,$$
and observe that the divisor $L^{(m)} - K_{Z^{(m)}/T} - \Delta^{(m)} +  Z_t^{(m)}$ is the sum of a nef and $f^{(m)}$-big divisor with the the pull-back of an ample divisor from $T$, hence it is nef and big. We can now apply Nadel vanishing as phrased in \autoref{lem:Nadel}\autoref{lem:Nadel2}. 
\end{proof}

The following is a higher dimensional version of \autoref{p_nef_Weil} and it will be used in the proof of \autoref{thm:bigness}. When $L$ is Cartier, it follows from \cite[Proposition 8.14]{KP}.

\begin{proposition}\label{prop:weakly-positive}
Let $f \colon (Z, \Delta) \to B$  be a locally stable family,  where $(Z,\D)$ is a  klt pair and $B$ is a smooth variety. Let $L$ be a $\mathbb{Q}$-Cartier $\Z$-divisor on $Z$ such that $L- K_{Z/T} - \Delta$ is $f$-big and nef.  
  Then, $f_* \mathcal{O}_Z(L)$ is a weakly-positive sheaf (in the sense of \cite[Definition 4.7]{KP}). 
\end{proposition}

\begin{proof}
Let $A$ be a very ample divisor on $B$ and $m$ a positive integer. 
Consider the $\Q$-Cartier $\Z$-divisor 
$$N:=L^{(m)} + \left( f^{(m)} \right)^* (K_B + A)\,.$$
on the fibre product $f^{(m)}: Z^{(m)} \to B$. Note that $(Z^{(m)}, \Delta^{(m)}) \to B$ is a locally stable family (see the proof of \cite[Proposition 2.12]{BHPS13} and since $B$ is smooth, we can apply \cite[Corollary 4.56]{Kobook} to show that $(Z^{(m)}, \Delta^{(m)})$ is klt.  Let $A_1 \in |A|$ be general. By \cite[Corollary 5.25]{KM}, we have an exact sequence

\begin{equation}\label{eq:short_exact_klt}
0 \to \O_{Z^{(m)}}\left(N + {f^{(m)}}^*A_1-Z^{(m)}_{A_1}\right) \to \O_{Z^{(m)}}(N+{f^{(m)}}^*A_1) \to \O_{{Z^{(m)}_{A_1}}}(N+{f^{(m)}}^*A_1) \to 0.
\end{equation}

Since $N+{f^{(m)}}^*A_1-Z_{A_1}^{(m)}-K_{Z^{(m)}} - \Delta^{(m)}$ is $f^{(m)}$-big and nef, by Kodaira vanishing \cite[Theorem 2.70]{KM} we get a surjection
$$
H^0(Z^{(m)}, \O_{Z^{(m)}}(N+{f^{(m)}}^*A_1)) \to H^0(Z^{(m)}_{A_1}, \O_{{Z^{(m)}_{A_1}}}(N+{f^{(m)}}^*A_1)). 
$$

Let $A_i \in |A|$ be general for $i=1,\ldots,d=\dim B$. Applying the argument above inductively we end up with a surjection

$$
H^0(Z^{(m)},  \O_{Z^{(m)}}(N+\sum_{i=1}^d{f^{(m)}}^*A_i)) \to  H^0(Z^{(m)}_b,  \O_{Z^{(m)}}(L^{(m)})|_{b}),
$$
where $b \in \cap_{i=1}^d A_i$.
The left hand side is isomorphic to
$$
H^0\left(B, \O_B(K_B + A + \sum_{i=1}^d A_i) \otimes f_*^{(m)} \O_{Z^{(m)}}(L^{(m)})\right)
$$
while the right hand side can be identified with $f_*^{(m)} \O_{Z^{(m)}}(L^{(m)}) \otimes k(b)$. We conclude that 
$$
\O_B(K_B+ A + \sum_{i=1}^d A_i) \otimes f_*^{(m)} \O_{Z^{m}}(L^{(m)}) \cong \O_B(K_B+ A + \sum_{i=1}^d A_i) \otimes \bigotimes_{j=1}^m f_*\O_Z(L)
$$
is generically globally generated for all $m >0$. 
Hence $\O_B(K_B+ A + \sum_{i=1}^d A_i) \otimes \Sym^{m}f_*\O_Z(L)$ is also generically globally generated, which implies that $f_*\O_Z(L)$ is weakly positive by definition.

\end{proof}

\section{The Berman-Gibbs-Viehweg (BGV) divisor}\label{ss:BGV divisor}

Let $(F,\Delta)$ be a normal  projective pair, and $L$ be a $\Q$-Cartier divisor on $F$. For every $q$ such that $qL$ is Cartier, let $r(q)=h^0(F,qL)$; we always assume $r(q)\geq 1$ to avoid trivial cases. The $q$-th Berman-Gibbs-Viehweg (BGV) divisor $\mathcal{D}^{\{q\}}$ is a Cartier divisor on the product $F^{\times r(q)}$ in the linear system $|qL^{\boxtimes r(q)}|$ defined in the following way. There is a canonical isomorphism between $H^0(F^{\times r(q)},qL^{\boxtimes r(q)})$ and $H^0(F,qL)^{\otimes r}$, and a canonical embedding of the one dimensional vector space $\det\left(H^0(F,qL)\right)$ into $H^0(F,qL)^{\otimes r(q)}$. The divisor $\mathcal{D}^{\{q\}}$ is then defined by the image of a non-zero element of $\det\left(H^0(F,qL)\right)$.  We set
$$
\g_q(L; F,\Delta)=  \lct\left(\mathcal{D}^{\{q\}} ; F^{\times r(q)}, \Delta^{(r(q))}\right)  \,,
$$
and then
$$
\g(L; F,\Delta):= \liminf_{q \to \infty} q\g_q(L; F,\Delta).
$$
Here, the notation for log canonical threshold (lct) is the following (the same as in \cite[Section 8]{KP}): given a klt pair $(X,D)$  and $\Gamma$ a $\Q$-Cartier divisor on $X$, then 
$$
\lct(\Gamma; X,D):=\sup \{t \in \mathbb R \ | \ (X,D+t\Gamma) \text{ is log canonical}   \}.
$$

Observe that, if $(F,\Delta)$ is klt, then $\g_q>0$ for every $q$.

Using \cite[Def. 8.9 and Cor. 8.12]{KP}, we have
$$
\g_q(L; F,\Delta) \ge \lct(\O_F(qL); F,\Delta).
$$
Recall that the alpha invariant $\alpha(L; F,\Delta)$ is defined as the infimum of the log canonical thresholds of all effective divisors $\Q$-linearly equivalent to $L$. 

\begin{lemma}\label{gamma_positive}
If $(F,\Delta)$ is klt and $L$ is ample, then $\g(L; F,\Delta)$ is strictly greater than zero.
\end{lemma}
\begin{proof}
 By \cite[Theorem 9.14]{Bou}, $\alpha(L; F,\Delta)$ is strictly greater than zero and so we have $0<\alpha(L;F,\Delta)=\liminf_{q}q\lct(\O_F(qL); F,0)\leq \g(L; F,\Delta)$.
\end{proof}

For the sake of completeness, let us briefly review some of the literature where these invariants have been already discussed. The BGV divisor plays a key role in the definition of Berman-Gibbs stability \cite{B,F,FO}. In the canonically polarized case, it can be used to construct the K\"{a}hler-Einstein metric. In the Fano case, it is used to define Berman-Gibbs stability, which is conjecturally equivalent to K-stability. 
In \cite{F,FO}, the definition of the gamma invariant is slightly different, and because of this we have labeled our invariant $\g$ rather than $\gamma$. The difference is that in loc. cit., the log canonical threshold is computed only along the small diagonal rather than on all $F^{\times r(q)}$. We do not know if the two different definitions give the same invariant. Also in loc.cit. the boundary is not considered.

To best of our knowledge, the gamma invariant has been computed only for $(\P^1,\O(1))$; in \cite{F}, it is shown that $\alpha=1/2$ and $\gamma=1$.

\section{Lower bound on the smallest slope of the Hodge bundles}\label{S:mumeno}

The following bound holds under the assumption that all the fibers of the family are reduced; it is used in the Viehweg product trick to guarantee that the total space of the fiber product is normal.

\begin{theorem}[Lower bound on the smallest slope of the Hodge bundle]\label{thm:lowermu_basic}
Let $f: (X,\D) \to T$ be a family where $T$ is a smooth projective curve, $(X,\D)$ is a projective pair, all fibers of $f$ are reduced, the general fibre is klt and $K_{X/T}+\D$ is $\Q$-Cartier and $f$-ample. For any integer $q>0$, let  $\mu_-(\E_q)$ the smallest slope of the Harder-Narasimanhan filtration of the Hodge bundle $\E_q=\E_{q,f}^{\D}$.

Fix $t \in T$ such that $(X_t,\Delta_t)$ is klt. With the notations of   \autoref{ss:BGV divisor}, consider the BGV-invariant
$$
\g_q:=\g_q(K_{X_t}+\Delta_t; X_t,\Delta_t)\,,\textrm{and  } \g:=\g(K_{X_t}+\Delta_t; X_t,\Delta_t)
$$
Let $q, \q \ge 2$ be positive integers such that  $q(K_{X/T}+\Delta)$  is an  Weil $\Z$-divisor and $\q(K_{X/T}+\Delta)$ is Cartier. Assume that both $\rk(\E_{q})$ and $\rk(\E_{\q})$ are strictly positive.

Then 
\begin{equation}\label{e_boundmumeno}
\mu_-\left(\E_{q}\right)\geq \min \left\{ \frac{q-1}{\q}, \frac{q\g_{\q}}{1 +  \g_{\q}\q }\right \}\frac{\deg{\lambda_{\q}}}{\rk \E_{\q}}.
\end{equation}

Moreover,
\begin{equation}\label{e:asym_bound}    
\mu_-\left(\E_{q}\right)\geq \min \left\{ q-1, \frac{q\g}{1 + \g} \right\}\frac{\deg{\lambda_{CM}}}{(n+1)(K_{X_t}+\Delta_t)^n}.
\end{equation}

In particular, if $f$ has maximal variation, then $\mu_-(\E_{q}) >0$, $\E_{q}$ and its determinant $\lambda_{q}$ are ample. 

\end{theorem}

\begin{proof}

Let $r$ be the rank of $\E_{\q}$. When $r>0$, we have a natural embedding
$\lambda_{\q} \hookrightarrow \E_{\q}^{\otimes r}$.
As $X$ is normal and all fibers are reduced, by \autoref{lem:normal} also the fiber product $X^{(r)}$ is normal, and we can apply K\"{u}nneth formula from \autoref{lem:kunneth}. Viehweg product trick is the observation that the above embedding together with the isomorphism 
$$
\E_{\q}^{\otimes r} \cong f_*^{(r)}\O_{X^{(r)}}\Big(\q\big(K_{X^{(r))}/T}+\D^{(r)}\big)\Big)
$$
 permits to write 
$$f^{(r)*}\lambda_{\q}+\Gamma^{\{\q\}}\sim \q\big(K_{X^{(r)}/T}+\D^{(r)}\big),$$
for some effective divisor $\Gamma^{\{\q\}}$ on $X^{(r)}$.  

Let $\ell$ be a positive rational number strictly smaller than $\min \{ \frac{q-1}{\q}, \frac{q\g_{\q}}{1 + \g_{\q}\q}\}$. We can take $\ell>0$ bcause, since the generic fiber is klt, $\g_q>0$. We will show that \begin{equation}\label{eq:bound_with_ell}
\mu_-\left(\E_{q}\right)\geq \ell \frac{\deg{\lambda_{\q}}}{\rk \E_{\q}}.
\end{equation}
Let $\tau \colon T' \to T$ be a degree $d$ surjective morphism from a smooth curve $T'$ whose branch divisor is supported in the locus where the fibers of $f$ are slc. Because of \autoref{prop:flat_base_change}, if we prove Equation \autoref{eq:bound_with_ell} after the base change via $\tau$, we obtain the result for the original family. After base change, the degree of $\lambda_{\q}$ gets multiplied by $d$, so if we choose a $d$ divisible enough, the degree of $\ell \tau^* \lambda_{\q}$ is an integer. We can thus assume that $\ell \lambda_{\q}$ is Cartier, and not only $\Q$-Cartier.

Consider the integral Weil  divisor
$$
L_{\q}=q\left(K_{X^{(r)}/T}+\Delta^{(r)}\right)- \ell f^{(r)*}\lambda_{\q}.
$$
First, observe that $L_{\q}$ is $\mathbb{Q}$-Cartier and $f^{(r)}$-ample. We can write
$$
L_{\q} \sim_\Q (q-\ell \q)\left( K_{X^{(r)}/T}+\Delta^{(r)}+\frac{\ell}{q-\ell \q}\Gamma^{(\q)}\right).
$$
The assumption 
$$
\ell \le\frac{q\g_{\q}}{1 + \g_{\q}\q}
$$
is equivalent to
$$
\frac{\ell}{q - \ell \q} \le \g_{\q}.
$$
We also have $0< q -\ell \q$, so 
\begin{equation*}
0 < \frac{\ell}{q - \ell \q}  \le \g_{\q}.
\end{equation*}
The divisor $\Gamma^{\{\q\}}_t$ on $X_t^{\times r(\q)}$ is the $\q$-th BGV divisor of $(X_t,\Delta_t, K_{X_t} + \Delta_t)$,  hence the pair 
\begin{equation*}
    \left(X_t^{\times r},\D_t^{(r)}+ \frac{\ell}{q-\ell \q}\Gamma^{\{\q\}}_t\right)
\end{equation*}
is klt.
We would like to apply \autoref{p_nef_Weil} to $L_{\q}$, so we consider the divisor
$$
L_{\q}-K_{X^{(r)}/T}-\D^{(r)}-\frac{\ell}{q-\ell \q}\Gamma^{\{\q\}}
=(q-\ell \q-1)\left( K_{X^{(r)}/T}+\Delta^{(r)}+\frac{\ell}{q-\ell \q}\Gamma^{\{\q\}}\right)
$$
The assumption $\ell \le (q-1)/\q$ gives $q-\ell \q >1$, in particular the above divisor is a positive multiple of $M:=\left( K_{X^{(r)}/T}+\Delta^{(r)}+\frac{\ell}{q-\ell \q}\Gamma^{\{\q\}}\right)$. Since $L_{\q}$ is $f^{(r)}$-ample, the same is true for $M$. By  \autoref{cor:semi_posit_lambda_e_CM}, $M$ is nef. We can thus apply Proposition \ref{p_nef_Weil} to show that $f_*^{(r)}\O_X(L_{\q})$ is nef.

By the projection formula, we have 
$$
f_*^{(r)}\O_{X^{(r)}}(L_{\q})=f_*^{(r)}\O_{X^{(r)}}\Big(q\big(K_{X^{(r)}/T}+\Delta^{(r)}\big)\Big)\otimes \O_T( -\ell \lambda_{\q})\,,
$$
and by  \autoref{lem:kunneth}, we have
$$
f_*^{(r)}\O_{X^{(r)}}\Big(q\big(K_{X^{(r)}/T}+\Delta^{(r)}\big)\Big)=\left(\E_{q}\right)^{\otimes r}\,.
$$
Since $\rk(\E^{(q)})>0$, we have $\rk\left(f_*^{(r)}\O_{X^{(r)}}(L_{\q}) \right)>0$;  because of Hartshorne characterization of nef vector bundles, we have $\mu_-(f_*^{(r)}\O_X(L_{\q}))\geq 0$, and
$$
0\leq \mu_-\left(f_*^{(r)}\O_{X^{r}}(L_{\q})\right)
=
\rk\left(\E_{\q}\right)\mu_-\left(\E_{q}\right)-\ell \deg(\lambda_{\q}).
$$
We obtain Equation \autoref{e_boundmumeno}  letting $\ell$ tend to $\min \{ \frac{q-1}{\q}, \frac{q\g_{\q}}{1 + \g_{\q}\q}\}$.

Equation \autoref{e:asym_bound} can be obtained from Equation \eqref{e_boundmumeno} taking the limit on any subsequence of the $\q$'s satisfying the hypotheses of the Theorem and of \autoref{prop:lambda_and_CM}. To compute the limit, recall \autoref{prop:lambda_and_CM}, the definition of $\g$ as liminf of $\g_{\q}\q$, and the classical asymptotic Riemann-Roch for the the rank of $\E_{\q}$.

Assume now that $f$ has maximal variation. Then by  \autoref{thm:positivity_lambda_e_CM} we can take an integer $\q$ big and divisible enough so that $\lambda_{\q}$ is ample. In particular $\deg(\lambda_{\q})\geq 1$ since $\lambda_{\q}$ is Cartier, and by Equation \autoref{e_boundmumeno} we get $\mu_-(\E_{q})>0$.  Equivalently, we can combine Equation \eqref{e:asym_bound} with \autoref{gamma_positive} and \autoref{thm:positivity_lambda_e_CM} to obtain $\mu_-(\E_{q})>0$.

\end{proof}

\begin{remark}[Case $q=1$]\label{ex:q0=1}
When $q=1$, a statement like Theorem \ref{thm:lowermu_basic} can not hold. Let us give an example. Take $S$ a smooth projective surface with strictly positive irregularity $h^1(S,\O_S)$. Given a fibration $f\colon S \to T$, where $T$ is a smooth curve, the bundle $\E^{(1)}$ has a summand $\O_T^{\oplus h^1(S,\O_S)}$, hence $\mu_-(\E^{(1)})=0$. On the other hand, if the fibers of $f$ are of general type, still $\deg(\lambda_{\q})>0$ for $\q$ big enough.

\end{remark}

The following is an immediate application of Theorem \ref{thm:lowermu_basic} using the stable reduction \autoref{prop:stable_reduction}.

\begin{corollary}[Effective positivity of Hodge bundles over a curve]\label{cor:ample}
 Let  $f\colon (X,\Delta)\to T$ be a  fibration where $(X,\D)$ is a projective normal log-pair with general fibre klt and $T$ is a smooth curve. Assume that $K_{X/T}+\D$ is $f$-ample and that $f$ has maximal variation. Let $q\ge 2$ be an integer such that $q\D$ is integral. Then
 $$
 f_*\O_X(q(K_{X/T}+\D))
 $$
 is an ample vector bundle on $T$ whenever it is non-zero. In particular, under the same assumptions, $\lambda_{q}$ is ample.
\end{corollary}

\section{Positivity of lambda classes and Hodge bundles}

\begin{theorem}[Effective positivity of lambda classes]\label{thm:postivity_lambda}
Let $f\colon (X,\Delta)\to B$ be a stable family, where $B$ is a normal projective variety of dimension $d\geq 1$. Assume further that at least one fibre of $f$ is klt, and $f$ has maximal variation.  Let $q\geq 2$ be an integer such that $q\Delta$ is integral and $\rk(\E^{(q)})>0$. Then $\lambda_{q}$ is big. 

If furthermore all fibers of $f$ are klt and the family has finite isomorphism classes, $\lambda_{q}$ is ample. 

Moreover, with the notations of \autoref{moduli}, if $qI\subset \Z$, $q\geq 2$, $\rk(\E^{(q)})>0$, and  $\M$ parametrizes at least one klt pair, then $\lambda_q$ is big on $\M^{\nu}$ 
\end{theorem}
\begin{proof}
Pick $\q$ big and divisible enough such that $\rk(\E^{(\q)})\neq 0$, $\q\Delta$ is integral, and $\lambda_{\q}$ is big (this is possible because $f$ has maximal variation and we can apply \autoref{thm:positivity_lambda_e_CM}). Observe that both $\lambda_{q}$ and $\lambda_{\q}$ are Cartier by construction, and  nef by \autoref{cor:semi_posit_lambda_e_CM}

Let $B_{klt}$ be the open dense subset of $B$ over which the fibers are klt. The invariant $\g_{\q}$ is constructible on $B_{klt}$, so we can fix an $s$ such that
$$
0 < s \le \min \left\{ \frac{q-1}{\q}, \frac{q\g_{\q}}{1 + \g_{\q}\q} \right\}.
$$
for every $t$ in $B_{klt}$. Let $\varepsilon=s/\rk(\E^{(\q)})$. 

Fix the class of $[C]$ of a movable curve on $B$ (in particular, a representative of $[C]$ intersect $B_{klt}$). We can find a morphism $g\colon T \to B$, such that $T$ is a smooth projective curve, the class of $g(T)$ is a positive multiple of $[C]$. By \autoref{prop:base_change_stable_families} the base-changed family $h\colon (X_T,\Delta_T)\to T$ is stable and $g^*\E^{(q)}_f=\E^{(q)}_h$. Since the general fibre of $h\colon (X_T,\Delta_T)\to T$ is klt, we have that  $(X_T,\Delta_T)$ is klt by \cite[Proposition 2.15]{Kobook}.

We can then apply \autoref{thm:lowermu_basic} to $h\colon (X_T,\Delta_T)\to T$ and obtain 
\begin{equation} \label{e:bound_sulle_curve}
\mu_-(g^*\E^{(q)}_f)\geq \varepsilon \deg(g^*\lambda_{\q,f}).
\end{equation}
In particular,
$$
\left(\lambda_{q}-\varepsilon \lambda_{\q}\right)\cdot [C]\geq 0.
$$
Recall that on a normal variety the cone of movable curves is dual to the cone of pseudoeffective divisors (see \cite[Theorem 2.2]{BDPP} and the paragraphs afterward for the normal case). Since we can choose as $[C]$ the class of any movable curve, and $B$ is normal, we have that $\lambda_{q}-\varepsilon \lambda_{\q}$ is pseudoeffective. Since we have choose $\q$ such that $\lambda_{\q}$ is big,  $\lambda_{q}$ is big as well.

If all fibers of $f$ are klt, we can make the same argument with any effective curve $[C]$, not just movable (the only place where we use that $[C]$ is movable is to guarantee that a representative intersect $B_{klt}$, but now we are assuming $B=B_{klt})$. This shows that $\lambda_{q}-\varepsilon \lambda_{\q}$ is nef. As the family has finite isomorphism classes, we can assume that $\lambda_{\q}$ is ample by Theorem \ref{thm:positivity_lambda_e_CM}, hence $\lambda_{q}$ is ample too.

Let us now prove the claim on $\M^{\nu}$.  We can prove it on the coarse space $M^{\nu}$. By \cite[Corollary 6.19]{KP}, we can take a finite normal cover $\pi\colon S\to M^{\nu}$ which is induced by a stable family over $S$. This family has maximal variation because $\pi$ is finite. Because of base change \autoref{prop:base_change_stable_families}, it is enough to show that $\lambda_{q}$ is big on $S$, but this is the previouse claim with $S=B$
\end{proof}

The following is an effective version of \cite[Theorem 8.1]{KP}.

\begin{theorem}[Bigness of Hodge bundles]\label{thm:bigness}
Let $f\colon (X,\Delta)\to B$ be a stable family over a normal base where $B$ is a normal projective variety of dimension $d\geq 1$. Assume further that at least one fibre of $f$ is klt, and $f$ has maximal variation.  Let $q\geq 2$ be an integer such that $q\Delta$ is integral and $\rk(\E^{(q)})>0$. Then $\E_{2q,f}^\D$ is big. 
\end{theorem}
\begin{proof}
By \autoref{thm:postivity_lambda}, we know that $\lambda_{q}$ is big.
Let $r:=\rk(\E^{(q)})$ and let $B_{klt}$ be the open subset of $B$ over which the fibres $(X_b, \D_b)$ are klt. 
By \cite[Lemma 8.10]{KP}, there exists $c >0$ such that $c < \lct(X_b,\D_b)$ for any $b \in B_{klt}$.
Set 
$$
\ell:= \lceil \frac{1}{c} \rceil.
$$

Let $\eta$ be a very ample divisor on $B$. Then for an integer $N$ big and divisible enough, $\lambda_{q}^{\otimes \ell} \otimes \eta^{\otimes N}$ is base point free and divisible by $\ell$. Then a ramified cover of degree $\ell$ along a general section of $\lambda_{q}^{\otimes \ell} \otimes \eta^{\otimes N}$ will be normal and we can assume that $\lambda_{q} = \O_B(\ell A)$ for some Cartier divisor $A$ on $B$. Taking a resolution, we can also assume that $B$ is smooth and so $(X,\D)$ klt by \cite[Corollary 4.56]{Kobook}.  In both cases, we applied the basechange \autoref{prop:base_change_stable_families}.

The fibre product $f: (X^{(r)},\D^{(r)}) \to B$ is a stable family and $(X^{(r)},\D^{(r)})$ is still klt.  Then the Viehweg's trick tells us that
$$
{f^{(r)}}^*\ell A + \Gamma \sim q(K_{X^{(r)}/B} + \D^{(r)}),
$$
where $\Gamma$ is an effective divisor on $X^{(r)}$. 
Note that $(X_b^{(r)}, \Gamma_b/\ell +\D^{(r)}_b)$ is klt for any $b \in U$ by the assumption on $\ell$. 
Then 
$$
f^{(r)}_* \O_{X^{(r)}}(2q (K_{X^{(r)}/B} + \D^{(r)}) - {f^{(r)}}^*A) \cong \bigotimes_{i=1}^r f_*\O_{X}(2q(K_{X/B}+\D)) \otimes \O_B(-A).
$$
is weakly-positive by \autoref{prop:weakly-positive} (let us stress that \autoref{thm:postivity_lambda} is used to have $A$ big).  Hence there exists a positive integer $b$ such that
\begin{align}
\O_B(bA) \otimes \Sym^{2b}\left(\O_B(-A) \otimes \bigotimes_{i=1}^r f_*\O_{X}(2q(K_{X/B}+\D)) \right) \cong \\
\cong \O_B(-bA) \otimes \Sym^{2b}\left( \bigotimes_{i=1}^r f_*\O_{X}(2q(K_{X/B}+\D)) \right) \xrightarrow[]{}\mathrel{\mkern-14mu}\rightarrow  \\
\xrightarrow[]{}\mathrel{\mkern-14mu}\rightarrow
  \O_B(-bA) \otimes \Sym^{2br}\left( f_*\O_{X}(2q(K_{X/B}+\D)) \right)
\end{align}
is generically globally generated and so $\E_{2q,f}^\D$ is big.

\end{proof}

\section{Lower bound on Chow-Mumford volumes for stable families}

\begin{theorem}[Lower bound on the Chow-Mumford volume for stable families] \label{thm:volumelambda}
Let $n$ be a positive integer and $\Lambda \subset \mathbb Q \cap [0,1]$ be a DCC set. Then there exists a constant $\delta=\delta(n,\Lambda) >0$ such that 
$$
\lambda_{CM}^{d}\geq \delta^{d}
$$
for every stable family $f\colon (X,\Delta)\to B$	satisfying the following conditions: $B$ is a normal projective variety of dimension $d\geq 1$, the coefficients of $\Delta$ are in $\Lambda$, the relative dimension of $f$ is $n$,  $f$ has maximal variation, and at least one fiber is klt.

In particular, if $d=1$, we obtain that
$$
(K_{X/B}+\D)^{n+1} \ge \delta.
$$

\end{theorem}

\begin{proof}

By \autoref{prop:base_change_stable_families}, we can assume that $B$ is smooth. In fact, let $g\colon B' \to B$ be a resolution of singularities, and let $f'\colon (X',\Delta')\to B'$ be the base-changed family. The family $f'$ is stable, has maximal variation and at least one fiber is klt because $g$ is birational. The coefficients of $\D'$ are in $\Lambda$ since $\D'$ is obtained as closure of the strict transform of $\D$ on a big open subset. In addition, since $\lambda_{CM,f'}=g^*\lambda_{CM,f}$ and $g$ is birational, the top self-intersection of the Chow-Mumford line bundle is preserved. From now on, we assume that $B$ is smooth.

Since $B$ is smooth, the general fibre of $f$ is normal and all fibres are reduced, we know that $X$ is normal and by \cite[Corollary 4.56]{Kobook} we conclude that $(X,\D)$ is klt
 
We now reduce to the case where the set of coefficients $\Lambda$ is finite.   First we take a small $\Q$-factorial  modification of $(X,\Delta)$.  Since such modification is small crepant, the volume is preserved and so we can assume $X$ $\Q$-factorial.  Thanks to \cite[Theorem 8.1]{Bir}, there exists a rational constant $\varepsilon>0$
depending just on $n$ and $\Lambda$ such that for every lc pair $(F,\Delta_F)$ of dimension $n$, coefficient of $\Delta_F$ in $\Lambda$, and $K_F+\Delta_F$ ample, $K_F+\varepsilon \Delta_F$ is big.  

We define a finite set $I=I(\Lambda,\varepsilon)$ as follow. Let $m$ be the minumum of $\Lambda$, which exists by the DCC condition. Let $J_i$ be the intervals $(\varepsilon^i, \varepsilon^{i-1} ] $, with $i>0$. We can cover $[m,1]$ with a finite number $r$ of intervals $J_i$'s. For each $i$, let $m_i$ be the minimum of the DCC set $\Lambda \cap J_i$ (if the intersection is empty, we do not define it). Define $I:=\{m_1,\dots , m_r\}$. We also introduce a map $\Lambda \to I$ which sends $a\mapsto a^-:=\max\{m_i \; \textrm{s.t.} \; m_i\leq a\}$.  Given a divisor $\Delta=\sum a_i \Delta_i$ with coefficients in $\Lambda$, we define $\Delta^-=\sum a_i^- \Delta_i$. Observe that $\varepsilon \Delta \leq \Delta^- \leq \Delta$, and $K_{X/B}+\Delta^-$ is $\Q$-Cartier because $X$ is $\Q$-factorial.

For $q$ multiple of the Cartier indexes of $K_{X/T}$, $\Delta$ and $\Delta^-$, we have an inclusion of Hodge bundles $\E_{q,f}^{\Delta^-}\subset \E_{q,f}^{\Delta}$. 
Because of the above quoted result \cite[Theorem 8.1]{Bir}, for a fibration as in the statement of Theorem we are proving, $K_{X/B}+\Delta^-$ is $f$-big. By \cite[Theorem 11.28 (i)]{Kobook}, we can now run an $f$-MMP to obtain a canonical model $f'\colon (X',\Delta')\to B$, which now is a stable family. Recall that, since the base $B$ is smooth, for $q$ divisible enough, $\O_{X'}(qK_{X/T})=\O_{X'}(qK_{X'})\otimes (f')^{*}\O_{B}(-qK_B)$; by the property \cite[Definition 11.26.5]{Kobook} of the canonical model, the Hodge bundles $\E_{q,f}^{\Delta^-}$ and $\E_{q,f'}^{\Delta'}$ are isomorphic for $q$ divisible enough, and hence we have an inclusion $\E_{q,f'}^{\Delta'}\subset \E_{q,f}^{\Delta}$. As both bundles are nef by \autoref{cor:semi_posit_lambda_e_CM}, and the quotient of nef bundles are nef, taking the determinant we have an inequality $\lambda_{q,f'}^{\Delta'}\leq \lambda_{q,f}^{\Delta}$. By \autoref{prop:lambda_and_CM} $\lambda_{CM,f'}^{\Delta'}\leq \lambda_{CM,f}^{\Delta}$ so we can replace $f$ with $f'$, and hence replace $\Lambda$ with the finite set $I$ defined above.

Fix an integer $q\geq 2$ such that $qI\subset \mathbb{Z}$, and for every klt pair $(Z,\Gamma)$ such that $K_Z+\Gamma$ is big, $\dim Z=n$ and the coefficient of $\Gamma$ are in $I$, the divisor $q(K_Z+\Gamma)$ gives a birational map. This $q$ exists by \cite[Theorem 1.3]{HMX}.	Let $\delta(n,I)={q}^{-n-1}$.

For the sake of simplicity, let us first discuss the case where $B$ is a curve. By \autoref{cor:ample} $\E_{q}$ is ample, in particular $\deg(\E_{q})\geq 1$. The claim follows for the slope inequality \autoref{slopeinequality}, $L=q(K_{X_T/T}+\Delta_T)$ ($L$ is nef by  \autoref{cor:semi_posit_lambda_e_CM}, $f_*\O_X(L)$ is nef by \autoref{thm: Fujino}), which asserts that
$$
(K_{X/T}+\Delta)^{n+1}\geq \deg(\E_{q})q^{-n-1}.
$$
(Recall that when the base is a curve $\deg \lambda_{CM}= (K_{X/B}+\D)^{n+1}$.)

Let us now consider a stable family over a base $B$ of dimension $d$. Let $[C]$ be the class of a movable curve in $B$. Take $g\colon T \to B$ as in the proof of Theorem \ref{thm:postivity_lambda}, let $h\colon (X_T,\Delta_T)\to T$ be the base change (recall we are in the situation of the base change \autoref{prop:base_change_stable_families}). Then \autoref{slopeinequality}, $L=q(K_{X_T/T}+\Delta_T)$ ($L$ is nef by  \autoref{cor:semi_posit_lambda_e_CM}, $f_*\O_X(L)$ is nef by \autoref{thm: Fujino}), implies that 
$$
(K_{X/T}+\Delta_T)^{n+1}\geq \frac{ \deg \lambda_{q,h}}{q^{n+1}}.
$$
We can apply \autoref{thm:postivity_lambda} and get $\deg(\lambda_{q,h})\geq 1$. Since all fibers of $f$ are slc, by the base change \autoref{prop:base_change_stable_families} $\lambda_{q,h}=g^*\lambda_{q,f}$ and $(K_{X_T/T}+\Delta_T)^{n+1}=g^*\lambda_{CM,f}$ .  Wrapping up we have
$$
\left(\lambda_{CM}-\delta\lambda_{q}\right)\cdot [C]\geq 0
$$
for any movable curve $[C]$. As $B$ is normal, $\lambda_{CM}-\delta\lambda_{q}$ is equal to a pseudoeffective divisor $P$. The divisors $\lambda_{CM}$ and $\lambda_{q}$ are nef ( \autoref{cor:semi_posit_lambda_e_CM}), and $\lambda_{q}$ is big ( \autoref{thm:postivity_lambda}), hence $\lambda_{CM}^k\lambda_{q}^{d-k-1}P\geq 0$ for all $k=1,\dots ,d-1$, hence
$$
\lambda_{CM}^{d}=\lambda_{CM}^{d-1}(\delta\lambda_{q} + P) \ge \lambda_{CM}^{d-1}\delta\lambda_{q} \ge \cdots \ge \delta^{d} \lambda_{q}^{d}.
$$
We conclude observing that $\lambda_{q}$ is Cartier, nef ( \autoref{cor:semi_posit_lambda_e_CM}), and big (\autoref{thm:postivity_lambda}), hence $\lambda_{q}^{d}\geq 1$. 

The \emph{in particular} part, follows since $\deg \lambda_{CM}= (K_{X/B}+\D)^{n+1}$ if the base is a curve. 
\end{proof}

The proof of the following Corollary is a variant of the proof of Theorem \ref{thm:volumelambda}.

\begin{corollary}[Lower bound on the Chow-Mumford volume for moduli spaces] \label{cor:volume_moduli}
Let $n$ be a positive integer, $\Lambda \subset \mathbb Q \cap [0,1]$ be a DCC set closed under addition. Then there exists a constant $C=C(n,I) >0$ with the following property. For every irreducible component $M$ of the coarse moduli space of $n$-dimensional stable pairs with coefficients in $\Lambda$ such that at least one parametrized pair is klt, we have the inequality
    $$
    ( \lambda_{CM})^{\dim M}\geq (vC)^{-\dim M}
    $$
 where  $v$ is the volume of the pairs parametrized by $M$, and $\lambda_{CM}$ is the $\Q$ divisor from \autoref{moduli}.
\end{corollary}
\begin{proof}
Let $\nu \colon M^{\nu}\to M$ be the normalization of $M$. It is enough to prove the statement on $M^{\nu}$.
 Let $[C]$ be the class of a movable curve on $M^{\nu}$. Then there exists a finite map $g\colon T\to M^{\nu}$ from a smooth projective curve $T$ such that

\begin{itemize}

    \item the class of $g(T)$ is a positive multiple of $[C]$;
    \item there is a family of $n$-dimensional stable pairs $h\colon (X,\Delta)\to T$ with coefficients in $I$  such that $g\circ \nu $ is the associated moduli map;
    \item the generic fiber of $h$ is klt (this is because $[C]$ is movable and at least one pair parametrized by $M$ is klt) and all fibers are reduced, in particular $(X_T,\Delta_T)$ is klt.
\end{itemize}

Let $I=I(\Lambda,n)$, $q$ and $\delta$ as in the proof of \ref{thm:volumelambda}. On $M^{\nu}$ consider the $\Q$-line bundle $L=\lambda_{CM}-q^{-n-1} \lambda_{q}$.

Because of base change \autoref{prop:base_change_stable_families}, $g^*L=\lambda_{CM,h}-q^{-n-1}\lambda_{q,h}$. Arguing as in the proof of Theorem \ref{thm:volumelambda}, we apply again \autoref{slopeinequality} and \autoref{thm:postivity_lambda} to show that $\deg(g^*L)\geq 1$, hence $L[C]>0$ and so $L$ is pseudoeffective on $M^{\nu}$. Again following the proof of Theorem \ref{thm:volumelambda}, we get $\lambda_{CM}^{d}\geq \delta^{d} \lambda_{q}^{d}$, where $d=\dim M^{\nu}$. Now $\lambda_{q}$ is nef and big by \autoref{cor:semi_posit_lambda_e_CM} and \autoref{thm:volumelambda}, but not Cartier. The Cartier index is a divisor of the number $A$ from \autoref{moduli}, so $\lambda_{q}^d\geq A^{-d}$.

To conclude the proof, observe that by \cite[Theorem 1.1]{HMX-Aut} the cardinality of the automorphism group of $(F,\Gamma)$, for every $(F,\Gamma)$ in $\M$, is bounded by the volume of $(F,\Gamma)$ times a constant which depends only on $n$ and $I$, so the same is true for $A$.

\end{proof}

We conclude this section with the following result about families over curves that are not necessarily stable. 

\begin{corollary}\label{cor:relative_volume}
Let $n$ be a positive integer and $\Lambda \subset \mathbb Q \cap [0,1]$ be a finite set. Then there exists a constant $\delta=\delta(n,\Lambda) >0$ such that 
$$
(K_{X/T}+\D)^{n+1} \geq \delta,
$$
for any fibration $f\colon (X,\Delta)\to T$ where $(X,\D)$ is a projective normal log-pair with general fibre klt and $T$ is a smooth curve such that $K_{X/T}+\D$ is $f$-ample and that $f$ has maximal variation.
\end{corollary}
\begin{proof}
Fix an integer $q\geq 2$ such that $qI\subset \mathbb{Z}$, and for every klt pair $(Z,\Gamma)$ such that $K_Z+\Gamma$ is big, $\dim Z=n$ and the coefficient of $\Gamma$ are in $I$, the divisor $q(K_Z+\Gamma)$ gives a birational map. This $q$ exists by \cite[Theorem 1.3]{HMX}.	Let $\delta(n,I)={q}^{-n-1}$.

 By \autoref{cor:ample} $\E_{q}$ is ample, in particular $\deg(\E_{q})\geq 1$. The claim follows for the slope inequality \autoref{slopeinequality}, $L=q(K_{X_T/T}+\Delta_T)$, which asserts that
$$
(K_{X/T}+\Delta)^{n+1}\geq \deg(\E_{q})q^{-n-1} \ge \frac{1}{q^{n+1}}.
$$
\end{proof}

\section{Automorphisms groups}

Let $(X,\D)$ be a slc pair. We denote by $\Bir(X,\D)$ (resp.\ $\Aut(X,\D)$) the birational automorphism  group  (resp.\ the regular automorphism group) of $(X,\D)$ (see \cite[Definition 6.15]{Fuj14}).
If $K_{X}+ \D$ is big, then $\Bir(X,\D)$ is finite by \cite[Thm 6.16]{Fuj14}. In fact, as a consequence of the DDC of the volume \cite[Theorem 1.3]{HMX}, the following holds in the same fashion as \cite[Theorem 1.1]{HMX-Aut} (see \cite[Theorem 2.8]{Ale} for the surface case). The result is well-known to expert, but we could not find it in the literature, so we write it down explicitly. 

\begin{theorem}\label{thm:aut} 
Let $n$ be a positive integer and let $\Lambda \subset [0,1]\cap \Q$ satisfying DCC. Then there exists a constant $C$ depending only on $n$ and $I$ such that
$$
|\Bir(X,\D)|\le C \vol(K_X + \D),
$$
for any $(X, \D)$ log canonical of dimension $n$ such that $K_X + \D$ is big, $\coeff(\D) \subset \Lambda$ and $(X, \D)$ admits a log canonical model.
\end{theorem}
\begin{proof}
Let $(X',\D')$ be the canonical model of $(X,\D)$. Then $G=\Bir(X,\D)=\Bir(X',\D')=\Aut(X',\D')$ is finite. By Lemma \ref{lem:quotient} taking the quotient by $G$ we obtain a log canonical pair $(Y,\Gamma)$ such that the coefficients of $\Gamma$ are in the DCC set 
$$
\Lambda'=\{0\le 1-(1-n_jb_j)/m_i \le 1 \ | \ b_j \in \Lambda \cup \{0\}, m_i, n_j\in \N \},
$$ 
and
$$
\vol(Y,\Gamma) =\frac{\vol(X,\D)}{|\Aut(X',\D')|}.
$$
The result follows now by \cite[Theorem 1.3]{HMX}.
\end{proof}

The following lemma is standard.

\begin{lemma}\label{lem:quotient}
Let $(X,\D=\sum b_j D)$ be a pair and $G \subset Aut(X,\D)$ be a finite subgroup. Set $\pi: X \to Y=X/G$. Then 
\begin{enumerate}
    \item there exists a divisor $\Gamma$ on $Y$ such that $K_X+\D=\pi^*(K_Y+\Gamma)$;
    \item $(X,\D)$ is klt (resp.\ lc, slc) iff $(Y,\Gamma)$ is klt (resp.\ lc, slc);
    \item the coefficients of $\Gamma$ are in the set
$$
\{0\le 1-(1-n_jb_j)/m_i \le 1 \ | \  m_i, n_j\in \N \}.
$$ 
\end{enumerate}
\end{lemma}
\begin{proof}
See \cite[Theorem 2.8]{Ale} for the surface case and \cite[2.41 and 2.42]{Kosing} for the general case.  
\end{proof}

Let $f:(X,\D) \to B$ be a fibration where $(X,\D)$ is a slc pair and $B$ is a normal variety. 
We denote by $\Aut(f) \subset \Aut(X,\D)$ the relative automorphism group, i.e.\  automorphisms of $(X,\D)$ preserving the fibration. This means that $\tau \in \Aut(f)$ is an automorphism of $(X,\D)$ such that $f \circ \tau = f$.

As a consequence of Theorem \ref{thm:volumelambda} we can prove a relative version of Theorem \ref{thm:aut}. 

\begin{corollary}\label{cor:autksb}
Let $n$ be a positive integer and $\Lambda \subset \mathbb Q \cap [0,1]$ be a DCC set. Then there exists a constant $\delta=\delta(n,\Lambda) >0$ such that 
$$
|\Aut(f)| \le \delta \vol(K_{X/B}+ \D)
$$
for any fibration $f\colon (X,\Delta)\to B$	satisfying the following conditions: the coefficients of $\Delta$ are in $\Lambda$, the relative dimension of $f$ is $n$, $f$ is a stable family with  maximal variation, at least one fiber is klt, and  $B$ is a smooth projective curve. 
\end{corollary}

\begin{proof}
First, observe that since the family is stable and the generic fiber is klt, then $(X,\Delta)$ is lc.

Since $\Aut(f) \subset \Aut(F,\D_F)$ where $F$ is a general fibre, we get that $\Aut(f)$ is finite. By Lemma \ref{lem:quotient}, taking the quotient $\pi: (X,\D) \to (Y,\Gamma)$ we obtain an lc pair $(Y,\Gamma)$ such that the coefficients of $\Gamma$ are in the DCC set 
$$
\Lambda'=\{0\le 1-(1-n_jb_j)/m_i \le 1 \ | \ b_j \in \Lambda\cup \{0\}, m_i, n_j\in \N \}
$$ 

with $K_X+\D = \pi^*(K_Y + \Gamma)$. By definition of $\Aut(f)$ we get a family $g:(Y,\Gamma) \to B$ such that $f^*K_B=\pi^*g^*K_B$ and so $K_{X/B}+\Delta=\pi^*(K_{Y/B}+\Gamma)$. Since $\pi$ is finite,  $K_{Y/B}+\Gamma$ is $g$-ample. Any fibre of $g$ is a finite quotient of a fibre of $f$ and so it is lc. This means that $g$ is also a stable family. We have
$$ 
(K_{X/B}+\Delta)^{n+1}=|\Aut(f)|(K_{Y/B}+\Gamma)^{n+1}.
$$
By Theorem \ref{thm:positivity_lambda_e_CM}, $g$ has maximal variation.
 
The result follows from Theorem \ref{thm:volumelambda}, and by 
$$
(K_{X/B}+\Delta)^{n+1}=|\Aut(f)|(K_{Y/B}+\Gamma)^{n+1}.
$$
\end{proof}

\section{Nefness threshold with respct to a fiber}

Given a family $f\colon X\to T$ of projective varieties over a smooth, projective curve, and an $f$-ample $\Q$-Cartier divisor $L$ on $X$, we can look at the nefness threshold with respect to a fiber
$$
\lambda_-(L)=\sup \{ t \in \R \, | \, L- tF \; \textrm{is nef} \}\,,
$$
where $F$ is the class of a fiber. We have the following general result, whose proof is a variation on \cite[Page 15]{Xu_Zhuang}.
\begin{lemma}\label{lem:lim}
With the above notation, let $k$ be the Cartier index of $L$, assume that all fibers of $f$ are reduced and at least one fiber is klt, then
$$\lim_{m\to \infty}\frac{\mu_-\left(f_*\O_X(mkL)\right)}{mk}=\lambda_-(L)$$
\end{lemma}
\begin{proof}
 We have the linearity property $\lambda_-(kL)=k\lambda_-(L)$, so we can replace $L$ with $kL$, and assume $k=1$ and $L$ Cartier.  Let $\E_m=f_*\O_X(mL)$. First, the sequence $\frac{\mu_-(\E_m)}{m}$ have a limit by \cite[Theorems 8.9, 8.11]{Chen}. Take a rational number $c<\lambda_-(L)$, so that $L-cF$ is ample. For $m$ big and divisible enough, $m(L-cF)$ is Cartier and $m(L-cF)-K_{X/T}-\Delta$ is ample, hence by \autoref{p_nef_Weil} the vector bundle $f_*\O_X(m(L-cF))$ is nef (we need that the fibers of $f$ are reduced and the general fiber is klt to apply \autoref{p_nef_Weil}, here is the only place where we use these two hypotheses). This means that $\mu_-(\E_m(-mc))\geq 0$, and so $\mu_-(\E_m)\geq qc$. Letting $c$ tend to $\lambda_-(L)$, we obtain that $\lambda_-(L)\leq \lim_{m\to \infty} \frac{\mu_-(\E_m)}{m}$. Take a rational number $c>\lambda_-(L)$, so that $L-cF$ is not nef. For any $m$ big and divisible enough, we have that $\O_X(m(L-cF))$ is an $f$-globally generated line bundle, since $L-cF$ is $f$-ample. Observe that the divisibility condition depend on $c$, because we need the line bundle to be Cartier. This means that $f^*f_*\O_X(mL-mcF) \to \O_X(mL-mcF)$ is surjective, which implies that $f_*\O_X(mL-mcF)$ is not nef and so $\mu_-(f_*\O_X(mL-mcF)) < 0$ by Hartshorne characterization of nef vector bundles on curves. This gives $\mu_-(\E_m) < cq$, and letting $c$ tend to $\lambda_-$ we obtain that $\lambda_-(L)\geq \lim_{m\to \infty} \frac{\mu_-(\E_m)}{m}$.
 
\end{proof}

We do not know whether the sequence $q^{-1}\mu_-(f_*\O_X(qL))$ admits a limit when $q$ tends to infinity if $L$ is only a $\Q$-Cartier divisor rather than a Cartier divisor.

\medskip

The following result could be compared with the nefness threshold \cite[Theorem 1.20]{CP}, see also \cite[Proposition 4.9]{Xu_Zhuang},\cite[Theorem 1.3]{CP23} and \cite[Theorem 3.4]{CP17}.

\begin{corollary}[Nefness threshold of the relative log-canonical bundle]\label{nef_threshold}
With the assumptions and notations of \autoref{thm:lowermu_basic}, except that we allow $f$ to have non-reduced fibers, we have
\begin{equation}\label{eq:slope}
\lambda_-(K_{X/T}+\Delta)\geq \frac{\g}{1+\g}\frac{\deg{\lambda_{CM}}}{(n+1)(K_{X_t}+\Delta_t)^n}
\end{equation}
\end{corollary}
\begin{proof}
Because of \autoref{lem:fibre_ridotte}, up to replacing $(X,\Delta)$ with $(Y,\Gamma)$ we can assume that all fibers of $f$ are reduced.
Let $k$ be the Cartier index of $K_{X/T}+\Delta$, and $m$ any positive integer greater than or equal to two, and such that $f_*\O_X(mk(K_{X/T}+\Delta))$ is not trivial. In equation \eqref{e:asym_bound} from \autoref{thm:lowermu_basic}, take $q_0=mk$, and divide both sides by $mk$. 

The limit for $m$ that goes to infinity of the lef hand side of Equation \autoref{eq:slope} is computed in \autoref{lem:lim}. 

Let us study the right hand side. By \autoref{gamma_positive}, $0<\frac{\g}{1+\g}<1$; for $m$ large enough we thus have
$$
\min \left\{mk-1, \frac{mk\g}{1 + \g}\right\}=\frac{mk\g}{1 +\g}
$$
 Therefore, for $m$ large enough, the right hand side is a linear function of $mk$, and the result follows.
\end{proof}

\section{New slope inequalities}
Broadly speaking, a slope inequality for a polarized family of varieties $f\colon (X,L)\to T$ over a curve $T$, is an inequality of the form $L^{\dim X} \geq C \deg f_*\mathcal{O}_X(L)$, where the constant $C$ depends only on the general fiber of $f$. Many slope inequalities, as well as detailed introduction to this topic, are presented in \cite{CTV}. Here, using our \autoref{thm: Fujino}, we can improve the slope inequalities for stable families obtained in \cite[Theorem B]{CTV}. 

First, \cite[Theorem 5.2 and Corollary 5.3]{CTV} can be replaced by our stronger \autoref{thm: Fujino} and its \autoref{cor:semi_posit_lambda_e_CM}.

\cite[Theorem B]{CTV} is a special case of \cite[Theorem 5.6]{CTV}, which in turn builds on \cite[Proposition 5.5]{CTV}. In the proofs of \cite[Proposition 5.5 and Theorem 5.6]{CTV}, we can replace all occurrences of \cite[Theorem 5.2 and Corollary 5.3]{CTV} with our \autoref{cor:semi_posit_lambda_e_CM}. In this way one can obtain various improvements, in particular we can remove the hypothesis ``$\Delta$ reduced" in \cite[Theorem B (1) and (3)]{CTV}, and ``$K_{X/T}+\Delta$ is nef" from \cite[Theorem B (4)]{CTV} obtaining the following result.

\begin{theorem}[Slope inequalities for stable families]
Let $f\colon (X,\Delta)\to T$ be a stable family over a a smooth, irreducible, projective curve $T$ and denote by $(F, \Delta_F)$ the general fiber of $f$. 
\begin{enumerate}
\item \label{TT:KSB-fam1}
Let $m \in \N_{>0}$ be an integer such that $m(K_{F}+\Delta_F)$  is Cartier and globally generated. Let $w\in \Q_{>0}$ such that the volume of the pull-back of $K_F+\Delta_F$ to any irreducible component of the normalization of $F$ is at least $w$. Then
$$ m^{n+1}(K_{X/T}+\Delta)^{n+1}\geq \frac{2wm^n}{wm^n+n}\deg \left(f_*\mathcal{O}_X(m(K_{X/T}+\Delta))\right).$$

\item \label{T:KSB-fam3} 
Let $m,q\in \N_{>0}$ such that at least one of the following conditions holds true 
\begin{itemize}
\item $\phi_{mq(K_{F} + \Delta_F)}$ is generically finite;
\item $mq(K_{F}+\Delta_F)$  is Cartier.
\end{itemize} 
Then 
$$
m^{n+1}(K_{X/T}+ \Delta)^{n+1}  \geq \frac{\deg f_*\O_X(m(K_{X/T}+\Delta))}{q^n}.
$$ 
\end{enumerate}
\end{theorem}

The above result, together with the arguments of  \cite[Section 5.2]{CTV} and \autoref{moduli}, gives  the following generalization of \cite[Theorem C]{CTV}.

\begin{theorem}[Ample cone of moduli spaces of stable pairs]\label{thm:ample_cone}
 Let $\M_{n,v,I}$ be the moduli space of slc pairs $(V,\Delta)$, such that the dimension of $V$ is $n$, the coefficients of $\Delta$ are in the finite set $I\subset [0,1] \cap \mathbb{Q}$ closed under addition, and the log canonical bundle is ample with volume $v$. Let $m$ be a positive integer  such that $m(K_V+\Delta)$ is Cartier and globally generated for every $(V,\Delta)\in \M_{n,v,I}(k)$ and let $w\in \Q_{>0}$ such that the volume of the pull-back of $K_V+\Delta$ to any irreducible component of the normalization of $V$ is at least $w$. Then the $\Q$-divisor
$$
\lambda_{CM}-\varepsilon \lambda_m
$$
is ample on the normalization of $\M_{n,v,I}$ for every rational number $\varepsilon$ in $\left[0,\frac{1}{m^{n+1}}\frac{2wm^n}{wm^n+n}\right)$.

\medskip

Consider two positive integers $m$ and $q$ such that, for every $(V,\Delta)\in \M_{n,v,I}(k)$, either $mq(K_V+\Delta)$ is Cartier or $\phi_{mq(K_V+\Delta)}$ is generically finite. Then the $\Q$-divisor
$$
\lambda_{CM}-\varepsilon \lambda_m
$$
is ample on the normalization of $\M_{n,v,I}$ for every rational number $\varepsilon$ in $\left[0,\frac{1}{q^{n}m^{n+1}}\right)$.
\end{theorem}


\begin{thebibliography}{a}

\bibitem[SP]{stackproject}  The Stacks project, \url{https://stacks.math.columbia.edu} , 2024 .

\bibitem[Ale94]{Ale} V. Alexeev, \emph{Boundedness and $K^2$ for log surfaces}, Internat. J. Math. 5 (1994), no. 6, 779--810.

\bibitem[ACSS]{ACSS} F. Ambro, P. Cascini, V. Shokurov, C. Spicer,\emph{Positivity of the Moduli Part}, arxiv preprint 2111.00423.


\bibitem[Ber18]{B} R.J. Berman, \emph{K\"{a}hler-Einstein metrics, canonical random point processes and birational geometry}, Proceedings of Symposia in Pure Mathematics, 97(1) (2018), pp. 29--73.

\bibitem[BDPP]{BDPP} S. Boucksom, J.P. Demailly, M. Păun, T. Peternell,
\emph{The pseudo-effective cone of a compact Kähler manifold and varieties of negative Kodaira dimension.}
J. Algebraic Geom. 22 (2013), no. 2, 201--248.

\bibitem[BG14]{BG14}
{\sc R.~J. Berman and H.~Guenancia}: \emph{K\"ahler-{E}instein metrics on
  stable varieties and log canonical pairs}, Geom. Funct. Anal. \textbf{24}
  (2014), no.~6, 1683--1730.


\bibitem[BHPS13]{BHPS13} B. Bhatt, W. Ho, Zs. Patakfalvi and C. Schnell, \emph{Moduli of products of stable varieties}, Compos. Math. {\bf 149} (2013), no.~12, 2036--2070; MR3143705

\bibitem[BZ16]{Bir} C. Birkar and D.-Q. Zhang, \emph{Effectivity of Iitaka fibrations and pluricanonical systems of polarized pairs}, Publ.math.IHES 123, 283--331 (2016).

\bibitem[BTJ]{Bou} S. Boucksom, T. Hisamoto and M. Jonsson \emph{Uniform K-stability, Duistermaat-Heckman measures and singularities of pairs}, Annales de l'Institut Fourier, Volume 67 (2017) no. 2, pp. 743-841.

\bibitem[Cas21]{Cas21} P. Cascini, \emph{New directions in the minimal model program}, Bollettino dell'Unione Matematica Italiana 14 (2021), 179--190.

\bibitem[CF14]{CF14} M. Corrêa, T: Fassarella, \emph{On the order of the automorphism group of foliations}, Math. Nachr. 287 (2014), no. 16, 1795–1803.

\bibitem[CM19]{CM19} M. Corrêa, A.Muniz, \emph{Polynomial bounds for automorphisms groups of foliations}, Rev. Mat. Iberoam. 35 (2019), no. 4, 1153--1194.

\bibitem[CS23]{CSMMP} P. Cascini and C. Spicer, \emph{MMP for algebraically integrable foliations}, London Math. Soc. Lecture Note Ser., 489 Cambridge University Press, Cambridge, 2025, 69--84.

\bibitem[Chen10]{Chen} H. Chen \emph{Convergence des polygones de Harder-Narasimhan}, Mém. Soc. Math. Fr. No. 120 (2010), 116 pp.

\bibitem[CHLX23]{CHLX} G. Chen, J. Han, J. Liu, and L. Xie \emph{Minimal model program for algebraically integrable foliations and generalized pairs}, Preprint arXiv:2309.15823

 \bibitem[CP17]{CP17} J. Cao and M. P\u{a}un, \emph{Kodaira dimension of algebraic fiber spaces over abelian varieties}, Invent.
Math. 207 (2017), no. 1, 345--387.

\bibitem[CP21]{CP} G. Codogni and Zs. Patakfalvi, \emph{Positivity of the CM line bundle for families of K-stable klt Fano varieties}, Inventiones Mathematicae 223 (2021), 811--894.

\bibitem[CP23]{CP23} G. Codogni and Zs. Patakfalvi, \emph{A note on families of K-semistable log-Fano pairs} Birational Geometry, Kaehler-Einstein Metrics and Degenerations Moscow, Shanghai and Pohang, Conference proceedings, Springer Proceedings in Mathematics and Statistics, volume 409, 2023


\bibitem[CTV21]{CTV1} G. Codogni, L. Tasin and F. Viviani, \emph{On some modular contractions of the moduli space of stable pointed curves},  Algebra Number Theory 15 (2021), no. 5, 1245--1281.

\bibitem[CTV23a]{CTV2} G. Codogni, L. Tasin and F. Viviani, \emph{On the first steps of the minimal model program for the moduli space of stable pointed curves}, Journal of the Institute of Mathematics of Jussieu,  2023, 22(1) 145--211.

\bibitem[CTV23b]{CTV} G. Codogni, L. Tasin and F. Viviani, \emph{Slope inequalities for KSB-stable and K-stable families}, Proceeding of the London Mathematica Society, Volume 126, Issue 4 2023, pp. 1394--1465.

\bibitem[Con00]{ConradBook} B. Conrard, \emph{Grothendieck duality and base change}, Lecture Notes in Mathematics 1750, Springer, 2000

\bibitem[Dru21]{Druel} S. Druel, \emph{Codimension 1 foliations with numerically trivial canonical class on singular spaces},
Duke Math. J. 170 (2021), no. 1, 95--203.


\bibitem[EV90]{EV90} H. Esnault, E. Viehweg, \emph{Effective bounds for semipositive sheaves and for the height of points on curves over complex function fields},  Compositio Math. 76 (1990), no. 1-2, 69--85.

\bibitem[EV91]{EV} H. Esnault, E. Viehweg, \emph{Ample sheaves on moduli schemes},  ICM-90 Satellite Conference Proceedings. Springer, Tokyo, 1991.

\bibitem[EV92]{EV92} H. Esnault and E. Viehweg: \emph{Lectures on vanishing theorems}, DMV
  Seminar, vol.~20, Birkh\"auser Verlag, Basel, 1992.

\bibitem[Fuj14]{Fuj14} O. Fujino, \emph{Fundamental theorems for semi log canonical pairs}, Algebraic Geometry 1 (2014), no. 2, 194--228.

\bibitem[Fuj16]{Fuj16} O. Fujino, \emph{Direct images of relative pluricanonical bundles}, Algebr. Geom. 3 (2016), no. 1, 50--62.


\bibitem[Fuj18]{Fuj} O. Fujino, \emph{Semipositivity theorems for moduli problems}, Ann. of Math. (2) 187(3): 639--665 (2018).

\bibitem[FF14]{FF14} O. Fujino, T. Fujisawa, \emph{Variations of mixed Hodge structure and semipositivity theorems}, Publ. Res. Inst. Math. Sci. 50 (2014), no. 4, 589--661.


\bibitem[FFS14]{FFS14} O. Fujino, T. Fujisawa, and M. Saito \emph{Some remarks on the semipositivity theorems}, Publ. Res. Inst. Math. Sci. 50 (2014), no. 1, 85--112.


\bibitem[Fuj78]{Fujita} T. Fujita, \emph{On K\"ahler fiber spaces over curves}, J. Math. Soc. Japan 30 (1978), no. 4, 779--794.

\bibitem[Fu16]{F} K. Fujita: \emph{On Berman-Gibbs stability and K-stability of $\Q$-Fano varieties by K. Fujita}, Compositio Mathematica, 152(2), 288--298 (2016).

\bibitem[FO18]{FO} K. Fujita,  Y. Odaka: \emph{On the K-stability of Fano varieties and anticanonical divisors}, Tohoku Math. J. (2) 70(4): 511--521 (2018). 

\bibitem[HL21]{HL21} C.D. Hacon, A. Langer, \emph{On birational boundedness of foliated surfaces}, J. Reine Angew. Math.770 (2021), 205--229.

\bibitem[HMX13]{HMX-Aut} C.D. Hacon, J. McKernan, C. Xu, \emph{On the binational automorphisms of varieties of general type}, Ann. of Math. 177(3): 1077--1111 (2013).

\bibitem[HMX14]{HMX} C. D. Hacon, J. McKernan, C. Xu, \emph{ACC for log canonical thresholds}, Ann. of Math. 180(2): 523--571 (2014).

\bibitem[HMX18]{HMX-bound} C. D. Hacon, J. McKernan, C. Xu, \emph{Boundedness of moduli of varieties of general type}, J. Eur. Math. Soc. 20 (2018), no. 4, pp. 865–901


\bibitem[HJLL24]{HJLL} J. Han, J. Jiao, M. Li, J. Liu,\emph{Volume of algebraically integrable foliations and locally stable families}, arxiv preprint 2406.16604.


\bibitem[Har80]{Har} R. Hartshorne, \emph{Stable reflexive sheaves}, Math. Ann. 254, 121--176 (1980).

\bibitem[Har94]{HarGenDiv} R. Hartshorne, \emph{Generalized divisors on Gorenstein Schemes}, K-Theory v. 8, n. 3 (1994).

\bibitem[H\"{o}r10]{Horing} A. H\"{o}ring, \emph{Positivity of direct image sheaves - a geometric point of view.}
L'Enseignement Mathématique 56, No. 1 (2010)

\bibitem[Kaw81]{Kaw81} Y. Kawamata, \emph{ Characterization of abelian varieties}, Compositio Math. 43 (1981), no. 2, 253--276.

\bibitem[Kaw82]{Kaw82} Y. Kawamata, \emph{Kodaira dimension of algebraic fiber spaces over curves}, Invent. Math. 66 (1982), no. 1, 57--71.

\bibitem[Kol86a]{Kol86a} J. Koll\'ar, \emph{ Higher direct images of dualizing sheaves. I}, Ann. of Math. (2) 123 (1986), no. 1, 11--42.

\bibitem[Kol86b]{Kol86b} J. Koll\'ar, \emph{ Higher direct images of dualizing sheaves. II},  Ann. of Math. (2) 124 (1986), no. 1, 171--202.


\bibitem[Kol87]{Kol87} J. Koll\'ar, \emph{Subadditivity of the Kodaira dimension: fibers of general type}, Algebraic geometry, Sendai, 1985, Adv. Stud. Pure Math., vol. 10, North-Holland, Amsterdam, 1987, pp. 361--398.

\bibitem[Kol90]{Kol90} J. Koll\'ar, \emph{Projectivity of complete moduli}, J. Differential Geom. 32 (1990), no. 1, 235--268.

\bibitem[Kol07]{Kol_resolution_of_singularities} J. Koll\'ar, \emph{Lectures on Resolution of Singularities}, Annals of Mathematics Studies, Volume 166, Princeton University Press 2007

\bibitem[Kol11]{Kol11} J. Koll\'ar, \emph{A local version of the Kawamata-Viehweg vanishing theorem}, Pure Appl. Math. Q. {\bf 7} (2011), no.~4, Special Issue: In memory of Eckart Viehweg, 1477--1494.

\bibitem[Kol13]{Kosing} J. Koll\'ar, \emph{Singularities of the Minimal Model Program}, Cambridge University Press (2013).
    
\bibitem[Ko118]{Kol18} J. Koll\'ar, \emph{Log-plurigenera in stable families.}, Peking Math. J. 2018, 1(1), 81--107.

\bibitem[Ko23]{Kobook} J. Koll\'ar, \emph{Families of Varieties of General Type}, Cambridge University Press (2023).

\bibitem[KM98]{KM} J. Koll\'ar and S. Mori, \emph{Birational Geometry of Algebraic Varieties}, Cambridge University Press (1998).


\bibitem[KMM87]{KMM87} Y. Kawamata, K. Matsuda, K. Matsuki, \emph{Introduction to the Minimal Model Problem}, Adv. Stud. Pure Math., 1987: 283--360.

\bibitem[KP17]{KP} S. Kovács, Zs. Patakfalvi, \emph{Projectivity of the moduli space of stable log-varieties and subadditvity of log-Kodaira dimension.} J. Amer. Math. Soc., 30(4):959--1021, 2017.

\bibitem[Laz2]{Laz2} R. Lazarsfeld, Positivity in Algebraic Geometry Vol. II, Ergebnisse der Mathematik und ihrer Grenzgebiete
49 (Springer, Berlin, 2004).
 
\bibitem[LPS20]{LPS20} L. Lombardi, M. Popa,  C. Schnell, \emph{Pushforwards of pluricanonical bundles under morphisms to abelian varieties}. J. Eur. Math. Soc. 22 (2020), no. 8, 2511--2536.

\bibitem[LS24]{LS24} L. Lombardi, C. Schnell, \emph{Singular hermitian metrics and the decomposition theorem of Catanese, Fujita, and Kawamata}. Proc. Amer. Math. Soc. 152 (2024), no. 1, 137--146.


\bibitem[Lu25]{Lu} X. L\"{u}, \emph{Unboundedness of foliated varieties} , International Journal of Mathematics Vol. 36, No. 6 (2025)

\bibitem[LT24]{LT24} X. L\"{u}, S. Tan, \emph{The Poincaré Problem for a foliated surface}, Preprin arXiv:2404.16293.


\bibitem[Mum12]{MumAbel} D. Mumford, \emph{Abelian Varieties}, Tata Institute of Fundamental Research Publications
Volume: 13; 2012; 263 pp.


\bibitem[Nak04]{Nak} N. Nakayama: \emph{Zariski-decomposition and abundance.} MSJ Memoirs, 14. Mathematical Society of Japan, Tokyo, 2004.

\bibitem[Pas24]{Passantino} A. Passantino, \emph{Numerical conditions for the boundedness of foliated surfaces}, Preprint arXiv:2412.05986.

\bibitem[Pat14]{Pat14} Zs. Patakfalvi, \emph{Semi positivity in positive characteristic} Annales scientifiques de l'ENS, vol 47, no 5, (2014)

\bibitem[Pat16]{Pat16} Zs. Patakfalvi, \emph{Fibered stable varieties}, Trans. Amer. Math. Soc. 368 (2016), no. 3, 1837--1869.

\bibitem[PS02]{PS02} J.V. Pereira, P.F. Sánchez, \emph{Transformation groups of holomorphic foliations}, Comm. Anal. Geom. 10 (2002), no. 5, 1115--1123.

\bibitem[PS19]{PS19} J.V. Pereira, R. Svaldi \emph{Effective algebraic integration in bounded genus}, Algebr. Geom. 6 (2019), no. 4, 454--485.

\bibitem[PM20]{Maciek} Z. Patakfalvi, M. Zdanowicz, \emph{On the Beauville--Bogomolov decomposition in characteristic $p\geq 0$}, arxiv 2020


\bibitem[PX16]{PX} Zs. Patakfalvi, C. Xu: \emph{Ampleness of CM line bundle on the moduli space of canonically polarized varieties}, Algebraic Geometry 4 (1) (2017) 29--39.

\bibitem[PS14]{PS} M. Popa, C. Schnell. \emph{On direct images of pluricanonical bundles}, Algebra Number Theory 8 (9) 2273 - 2295, 2014. 

\bibitem[ST23]{ST} T. Sano, L. Tasin ,\emph{On K-stability of Fano weighted hypersurfaces}. Algebr. Geom. 11 (2024), no. 2, 296--317.

\bibitem[SSW20]{Weil-Petersson} J. Song, J. Sturm, X. Wang: \emph{Continuity of the Weil-Petersson potential} Preprint arXiv:2008.11215
.

\bibitem[SS23]{SS23} C. Spicer, R. Svaldi, \emph{Effective generation for foliated surfaces: results and applications}, J. Reine Angew. Math. 795 (2023), 45--84.

\bibitem[Vie82] {Vie82} E. Viehweg, \emph{Die Additivität der Kodaira Dimension für projektive Faserräume über Varietäten des allgemeinen Typs.}, J. Reine Angew. Math. 330 (1982), 132--142.


\bibitem[Vie83] {Vie83} E. Viehweg, \emph{ Weak positivity and the additivity of the Kodaira dimension for certain fibre spaces}, Adv. Stud. Pure Math., 1 North-Holland Publishing Co., Amsterdam, 1983, 329--353.

\bibitem[Vie01] {Viehweg} E. Viehweg, \emph{Positivity of direct image sheaves and applications to families of higher dimensional manifolds.} ICTP-Lecture Notes 6 (2001).

\bibitem[XZ20]{Xu_Zhuang} C. Xu and Z. Zhuang: \emph{On positivity of the CM line bundle on K-moduli spaces}, Ann. of Math. (2) 192 (2020), no. 3, 1005–1068.


\end{thebibliography}
\end{document}